\documentclass[11pt]{amsart}

\usepackage{amsmath, amssymb, amsthm, amsfonts, enumerate, color, comment}
\usepackage[mathscr]{euscript}

\usepackage{palatino}
\usepackage{graphicx}

\usepackage{parskip}

\numberwithin{equation}{section}
\theoremstyle{plain}
\newtheorem{Proposition}[equation]{Proposition}
\newtheorem{Corollary}[equation]{Corollary}
\newtheorem*{Corollary*}{Corollary}
\newtheorem{Theorem}[equation]{Theorem}
\newtheorem*{Theorem*}{Theorem}

\theoremstyle{definition}

\newtheorem{Remark}[equation]{Remark}

\usepackage{enumitem}
\setlist[enumerate]{leftmargin=*}
\setlist[itemize]{leftmargin=*}

\def\A{\mathbb {A}}
\def\C{\mathbb{C}}
\def\R{\mathbb{R}}
\def\D{\mathbb{D}}
\def\T{\mathbb{T}}
\def\N{\mathbb{N}}

\def\M{\mathcal{M}}

\def\re{\mathop{\rm Re}\nolimits}

\renewcommand{\leq}{\leqslant}
\renewcommand{\geq}{\geqslant}
\renewcommand{\subset}{\subseteq}
\renewcommand{\phi}{\varphi}
\renewcommand{\vec}[1]{{\bf #1}}

\usepackage{xcolor}
\newcommand{\0}{{\color{lightgray}0}}

\author[E. A. Gallardo-Guti\'errez]{Eva A. Gallardo-Guti\'errez}
\address{Eva A. Gallardo-Guti\'errez \newline
Departamento de An\'alisis Matem\'atico y Matem\'atica Aplicada,\newline
Facultad de Matem\'aticas,
\newline Universidad Complutense de
Madrid, \newline
Plaza de Ciencias 3, 28040 Madrid,  Spain
\newline
and \newline
Instituto de Ciencias Matem\'aticas ICMAT (CSIC-UAM-UC3M-UCM),
\newline Madrid,  Spain }
\email{eva.gallardo@mat.ucm.es}

\author[J. R. Partington]{Jonathan R. Partington}
\address{Jonathan R. Partington, \newline
School of Mathematics, \newline
University of Leeds, \newline
Leeds LS2 9JT, United Kingdom}
\email{J.R.Partington@leeds.ac.uk}

\author[W. T. Ross]{William T. Ross}
\address{William T. Ross, \newline
Department of Mathematics and Statistics,
\newline 
University of Richmond, 
\newline Richmond, VA 23173, USA}
\email{wross@richmond.edu}
	

\thanks{First author is partially supported by Plan Nacional  I+D grant no. PID2019-105979GB-I00 (Spain), the Spanish Ministry of Science and Innovation, through the ``Severo Ochoa Programme for Centres of Excellence in R\&D'' (CEX2019-000904-S) and from the Spanish National Research Council, through the ``Ayuda extraordinaria a Centros de Excelencia Severo Ochoa'' (20205CEX001).}

\title{Invariant subspaces of the Ces\`{a}ro operator}


\begin{document}

\begin{abstract}
This paper explores various classes of invariant subspaces of the classical Ces\`{a}ro operator $C$ on the Hardy space $H^2$. We provide a characterization of the finite co-dimensional $C$-invariant subspaces, based on earlier work of the first two authors, and determine exactly which model spaces are $C$-invariant subspaces; using this, we describe the $C$-invariant subspaces contained in model spaces, which we show are all cyclic. Along the way, we re-examine an associated Hilbert space of analytic functions on the unit disk developed by Kriete and Trutt. We also make a connection between the adjoint of the Ces\`{a}ro operator and certain composition operators on $H^2$ which have universal translates in the sense of Rota.
\end{abstract}

\maketitle

\section{Introduction}

This paper examines the invariant subspaces of the {\em Ces\`{a}ro operator}
$$(C f)(z) := \frac{1}{z} \int_{0}^{z} \frac{f(\xi)}{1 - \xi} d\xi$$ on the {\em Hardy space} $H^2$ \cite{Duren} of the unit disk $\D := \{z \in \C: |z| < 1\}$. Recall that $H^2$ is the Hilbert space of analytic functions $f$ on $\D$ with
$$
\|f\|_{H^2} := \Big(\sup_{0 < r < 1} \int_{\T} |f(r \xi)|^2 dm(\xi)\Big)^{\frac{1}{2}} < \infty.
$$
In the above,  $m$ denotes Lebesgue measure on the unit circle $\T := \partial \D$, normalized so that $m(\T) = 1$. Note also that $\|f\|_{H^2}^{2} = \sum_{n \geq 0} |a_n|^2$, where $\{a_n\}_{n \geq 0}$ is the sequence of Taylor coefficients of $f$.

Hardy's inequality \cite{MR1544414} (see also \cite{MR944909}) says that $C$ is a bounded operator on $H^2$. Moreover, a calculation shows that the matrix representation of $C$ with respect to the orthonormal basis $\{z^n\}_{n \geq 0}$ for $H^2$, i.e., the infinite matrix whose entries are $\langle C z^{n}, z^m\rangle_{H^2} $, $m, n \geq 0$, is the well known {\em Ces\`{a}ro matrix}
\begin{equation}\label{Cedcecefcematrix}
C := \begin{bmatrix}
1 & \0 & \0 & \0 & \0 &  \cdots\\[3pt]
\frac{1}{2} & \frac{1}{2} & \0 & \0 & \0 & \cdots\\[3pt]
\frac{1}{3} & \frac{1}{3} & \frac{1}{3} & \0 & \0 & \cdots\\[3pt]
\frac{1}{4} & \frac{1}{4} & \frac{1}{4} & \frac{1}{4} & \0 & \cdots\\[3pt]
\frac{1}{5} & \frac{1}{5} & \frac{1}{5} & \frac{1}{5} & \frac{1}{5} & \cdots\\
\vdots & \vdots & \vdots & \vdots & \vdots & \ddots
\end{bmatrix}.
\end{equation}
In this matrix setting, one thinks of $C$ as acting on column vectors $\vec{a} = (a_0, a_1, a_2, \ldots)^{T}$ in $\ell^2$ via matrix multiplication
$$C \vec{a} = \Big(a_0, \frac{a_0 + a_1}{2}, \frac{a_0 + a_1 + a_2}{3}, \ldots\Big)^{T}.$$
As was known for quite some time, the Ces\`{a}ro matrix is the basis for an important summability method for Fourier series \cite{MR1963498}.
A more detailed analysis from a paper of Brown, Halmos, and Shields establishes the following basic operator theory facts about $C$ and its adjoint $C^*$ (which we record here for later use).

\begin{Proposition}[Brown--Halmos--Shields  \cite{MR187085}]\label{BrownShields}
For the Ces\`{a}ro operator $C$ on $H^2$,
\begin{enumerate}
\item[(a)] $\|I - C\| = 1$;
\item[(b)] $\|C\| = 2$;
\item[(c)] $\sigma(C) = \{z: |z - 1| \leq 1\}$;
\item[(d)] $\sigma_p(C) = \varnothing$;
\item[(e)] $\sigma_p(C^{*}) = \{z: |z - 1| < 1\}$.
\end{enumerate}
\end{Proposition}
Here, $\sigma(C)$ denotes the spectrum of $C$ and $\sigma_p(C)$ the point spectrum (eigenvalues of $C$).

The Ces\`{a}ro operator is hyponormal, namely, $C^{*} C - C C^{*} \geq 0$ \cite[Thm.~3]{MR187085} (see also \cite{MR1291108}). One of the gems in the study of the Ces\`{a}ro operator, and to be explored further in this paper,  is a theorem of Kriete and Trutt \cite{MR281025} which extends the above hyponormality  result and says that $I - C$ is unitarily equivalent to the operator $M_z f = z f$ (multiplication by the independent variable $z$) on $\mathcal{H}^2(\mu)$, where $\mu$ is a certain positive finite Borel measure on $\overline{\D}$, the closure of $\D$, and $\mathcal{H}^2(\mu)$ denotes the closure of the analytic polynomials $\C[z]$ in $L^2(\mu)$. This shows that the Ces\`{a}ro operator is subnormal (a normal operator restricted to one of its invariant subspaces).

A follow up paper of Kriete and Trutt \cite{MR350489} began a discussion of some of the complexities of the invariant subspaces of the Ces\`{a}ro operator via a discussion of the $M_z$-invariant subspaces of $\mathcal{H}^2(\mu)$. We find this space quite interesting and this paper will continue this line of inquiry.
A recent paper \cite{GP} of the first two authors of this current  paper rekindled the discussion of the complexities of the Ces\`{a}ro invariant subspaces from the vantage point  of semigroups of composition operators and of translation operators $f(x) \mapsto f(x - t)$ on a certain weighted $L^2$ space of the real line $\R$.

The purpose of this paper is to continue this invariant subspace discussion.  One particularly interesting class of invariant subspaces of the Ces\`{a}ro operator will be the model spaces $(u_{\alpha} H^2)^{\perp}$, where $\alpha > 0$ and $u_{\alpha}$ is the atomic inner function
\[
u_{\alpha}(z) = \exp\Big(\alpha \frac{z + 1}{z - 1}\Big).
\]
Using semigroup ideas from \cite{GP}, along with a Beurling--Lax theorem developed in that same paper, we will show in Theorem \ref{thm:ualphinv} that indeed  the model space $(u_{\alpha} H^2)^{\perp}$, where $\alpha > 0$, is an invariant subspace of the Ces\`{a}ro operator. Moreover, in Theorem \ref{thm:onlymodel} we will show that $(u_{\alpha} H^2)^{\perp}$, $\alpha > 0$, are the {\em only} model spaces which are Ces\`{a}ro invariant. One can also arrive at these two results using a discussion in \cite{MR2445578} involving the Volterra operator on $H^2$. In  Proposition \ref{cyclicCCC} we prove that the vector
$$\frac{1 - u_{\alpha}(z)}{1 + z}$$ belongs to $(u_{\alpha} H^2)^{\perp}$ and is a cyclic vector for the Ces\`{a}ro operator  when restricted to $(u_{\alpha} H^2)^{\perp}$.  In Proposition \ref{nosthnehrenaasdd} we prove that these model spaces $(u_{\alpha} H^2)^{\perp}$ correspond to a curious class of $M_z$-invariant subspaces of the associated Kriete--Trutt space $\mathcal{H}^2(\mu)$ that are different from the ``standard'' Beurling-type ones consisting of the closure of $uH^2$ in $\mathcal{H}^2(\mu)$ (as explored in \cite{MR350489}), where $u$ is an inner function.

An integral calculation will verify that the adjoint $C^{*}$ of the Ces\`{a}ro operator on $H^2$ is given by the integral formula
\begin{equation}\label{Cstar}
(C^{*} f)(z) = \frac{1}{1 - z} \int_{z}^{1} f(\xi) d \xi.
\end{equation}
A semigroup discussion in \cite{GP} characterized the finite dimensional $C^{*}$-invariant subspaces (equivalently the finite co-dimensional $C$-invariant subspaces) in the setting of translation invariant subspaces of a certain weighted $L^2$ space on $\R$. In Theorem \ref{thm:fdcstar} we use this discussion to recast this characterization in the $H^2$ setting and show that every finite dimensional $C^{*}$-invariant subspace is the span of finite unions of the functions
$$(1 - z)^{\mu} (\log(1 - z))^{j}, \quad 0 \leq j \leq k,$$
where $\Re \mu > -\tfrac{1}{2}$ and $k \in \N_0$ is fixed. We will prove this in \S \ref{firstysydydsysdf}, using some semigroup techniques from \cite{GP}  (see Theorem \ref{thm:fincodviakt} below for a different approach).

In \S \ref{Careadsdfds} we establish  a connection between the Ces\`{a}ro operator and the concept of a universal operator in the sense of Rota. In Theorem  \ref{uuuuniverereersal} we show that although $C^{*}$ is not universal, there is a bounded analytic function $F$ on the disk $\{z: |z - 1| < 1\}$ such that $F(C^{*})$ is universal. We prove this by using an interesting result from \cite{MR4373152} which says that certain linear translates $C_{\phi} - \lambda I$ of a class of composition operators $C_{\phi}$ on $H^2$ are universal in the sense of Rota.

Through the results in this paper, we hope to make the case that the Kriete--Trutt space $\mathcal{H}^2(\mu)$ is an important Hilbert space of analytic functions on $\D$ that is very much worthy of further study. It has a rich variety of ``nonstandard'' $M_z$-invariant subspaces and the overall complexity of these invariant subspaces is yet unknown.

\section{Semigroups of operators}\label{two}

For each  $t \geq 0$ let
$$\phi_{t}(z) := e^{-t} z + 1 - e^{-t}, \quad z \in \D.$$
These are analytic self maps of the open unit disk $\D$ and $\phi_{t}(\D)$ is an internally tangent disk at $\xi = 1$ (see Figure \ref{Figure_flow}). This collection of maps $\{\phi_t\}_{t \geq 0}$ defines a {\em holomorphic flow} in that
\begin{enumerate}
\item[(i)] $\phi_0(z) = z$ for all $z \in \D$;
\item[(ii)] $\phi_{t + s} = \phi_t \circ \phi_s$ for all $s, t \geq 0$ and all $z \in \D$; and
\item[(iii)] for any fixed $s \geq 0$ and any $z \in \D$, ${\displaystyle \lim_{t \to s} \phi_{t}(z) = \phi_s(z)}$.
\end{enumerate}
\begin{figure}
\begin{center}
 \includegraphics[width=.6\textwidth]{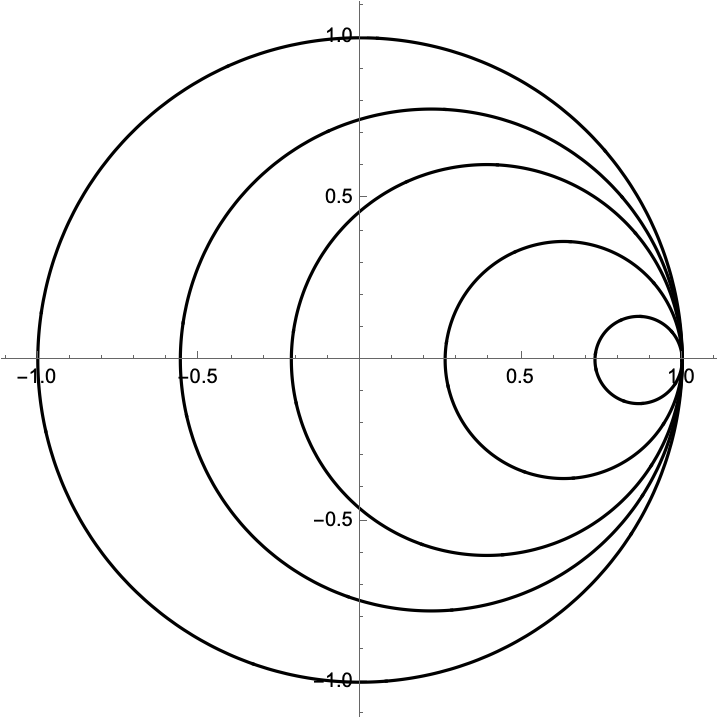}
 \caption{The circles $\phi_{t}(\T)$ for the maps $\phi_{t}(z) = e^{-t} z + 1 - e^{-t}.$ The circles get smaller as $t \to \infty$.}
 \label{Figure_flow}
 \end{center}
\end{figure}

The Littlewood subordination theorem \cite[p.~13]{MR1237406} implies that the composition operator
$$C_{\phi_t}: H^2 \to H^2, \quad C_{\phi_t} f = f \circ \phi_t,$$ defines a bounded linear operator which satisfies  $\|C_{\phi_t}\| = e^{t/2}$  \cite[Thm.~9.4]{MR1397026} and $\sigma(C_{\phi_t})  = \{z: |z| \leq e^{t/2}\}$ \cite[Thm.~7.26]{MR1397026}. In fact, one can quickly verify that the collection $\{C_{\phi_t}\}_{t \geq 0}$ defines a strongly continuous (or $C_0$) semigroup of operators on $H^2$ in that
\begin{enumerate}
\item[(i)] $C_{\phi_0} = I$;
\item[(ii)] $C_{\phi_t} C_{\phi_s} = C_{\phi_{s + t}}$ for all $s, t \geq 0$; and
\item[(ii)] ${\displaystyle \lim_{t \to 0^{+}} C_{\phi_t} f = f}$ for all $f \in H^2$.
\end{enumerate}
Using an analysis of the generator and co-generator of the $C_0$ semigroup $\{C_{\phi_t}\}_{t \geq 0}$, the authors in   \cite[Cor.~2.3]{GP} connect the above composition operators with the adjoint $C^{*}$ of the Ces\`{a}ro operator from \eqref{Cstar} by means of the following formula:
$$(C^{*} f)(z) = \int_{0}^{\infty} e^{-t} (C_{\phi_t} f)(z) dt.$$
The formula above has the following important consequence which is key to some of our results concerning the invariant subspaces of the Ces\`{a}ro operator.

\begin{Theorem}\label{wydfgviuc8d7s6ed7sfv6}
A closed subspace $\mathcal{M}$ of $H^2$ satisfies $C \mathcal{M} \subset \mathcal{M}$ if and only if $C_{\phi_t} \mathcal{M}^{\perp} \subset \mathcal{M}^{\perp}$ for all $t \geq 0$.
\end{Theorem}

We pause to mention the papers \cite{MR0733903, MR0897683, MR1021904}, which also use semigroups to obtain information about the Ces\`{a}ro operator. 

In a way, Theorem \ref{wydfgviuc8d7s6ed7sfv6}  is a Beurling--Lax theorem for the Ces\`{a}ro operator, reminiscent of the Beurling--Lax theorem for the invariant subspaces for the semigroup of shifts $f(x) \mapsto f(x - t)$, $t \in \R$, on $L^2(\R)$ \cite[p.~204]{MR3890074}.

There is another associated $C_0$ semigroup of shift operators from \cite{GP} which will play an important role in our discussion. Let $\C^{+}$ denote the right half plane
$$\C^{+}: = \{z: \Re z > 0\}$$ and let $H^2(\C^{+})$ denote the {\em Hardy space of $\C^{+}$} \cite{Garnett, MR3890074}. These are the analytic functions $F$ on $\C^{+}$ for which
$$\|F\|_{H^2(\C^{+})} := \Big(\sup_{0 < x < \infty} \int_{-\infty}^{\infty} |F(x + i y)|^2 dy\Big)^{\frac{1}{2}} < \infty.$$

\begin{Remark}
As to not be overly pedantic, we will use the notation $H^2$ (and not $H^2(\D)$) to denote the Hardy space of the disk $\D$. 
\end{Remark}

By a version of the Paley--Wiener theorem \cite[p.~203]{MR3890074}, every $F \in H^2(\C^{+})$ can be realized as
\begin{equation}\label{Laplace}
F(s) = (\mathcal{L} f) (s) := \frac{1}{\sqrt{2 \pi}} \int_{0}^{\infty} f(x) e^{-s x} dx, \quad s \in \C^{+},
\end{equation}
(the normalized Laplace transform) where $f \in L^2(\R_{+})$ (and $\R_{+} = (0, \infty)$). Furthermore,
$$\|F\|_{H^2(\C^{+})}^{2} =  \int_{0}^{\infty} |f(x)|^2 dx.$$
Conversely, $\mathcal{L} f \in H^2(\C^+)$ whenever $f \in L^2(\R_{+})$.
In other words, the normalized Laplace  transform $\mathcal L$  is a unitary operator from $L^2(\R_{+})$ onto $H^2(\C^{+})$.

The function
\begin{equation}\label{gammaC}
\gamma(z) = \frac{1 + z}{1 - z}, \quad z \in \D,
\end{equation} is a conformal map from $\D$ onto $\C^{+}$ which induces the unitary operator $\mathcal{U}: H^2 \to H^2(\C^{+})$ given by

\begin{equation}\label{UUU}
\mathcal{U} g(s)= \frac{1}{\sqrt\pi (1+s)}g \left(\frac{s-1}{s+1}\right), \qquad g \in H^2(\D)
\end{equation}
with
\begin{equation}\label{UUUinverse}
\mathcal{U}^{-1}G(z)= \frac{2\sqrt \pi}{1-z} G \left( \frac{1+z}{1-z}\right), \qquad G \in H^2(\C^+).
\end{equation}

Moreover, as discussed in \cite[Lemma 2.1]{MR2009562}, if $\psi$ is an analytic self map of $\C^{+}$ such that $C_{\psi} F = F \circ \psi$ defines a bounded composition operator on $H^2(\C^{+})$ (and this happens if and only if
$\psi(\infty) = \infty$ and
$$\lim_{z \to \infty} \frac{z}{\psi(z)}$$ exists and is finite) \cite{MR2966995, MR2390678}), then with
$$\Phi(z) := \gamma^{-1} \circ \psi \circ \gamma,$$
which defines an analytic self map of $\D$, the linear transformation
$$(\mathcal{L}_{\Phi} f)(z) := \frac{1 - \Phi(z)}{1 - z} f(\Phi(z))$$ defines a weighted composition operator on $H^2$ and
$$\mathcal{U}^{-1} C_{\psi} \mathcal{U} = \mathcal{L}_{\Phi}.$$
Applying this discussion to the family of functions defined on $\C^{+}$ by
$$\psi_{t}(s) = e^{t} s + (e^{t} - 1), \quad t \geq 0,$$
which are analytic self maps of $\C^{+}$ (see Figure \ref{Figure_Lines}) that induce bounded composition operators $C_{\psi_t}$ on $H^2(\C^{+})$, a discussion in \cite{MR4373152} shows that
$$\mathcal{U} C_{\phi_t} \mathcal{U}^{-1} = e^{t} C_{\psi_t} \quad \mbox{for all $t \geq 0$}.$$
\begin{figure}
\begin{center}
 \includegraphics[width=.6\textwidth]{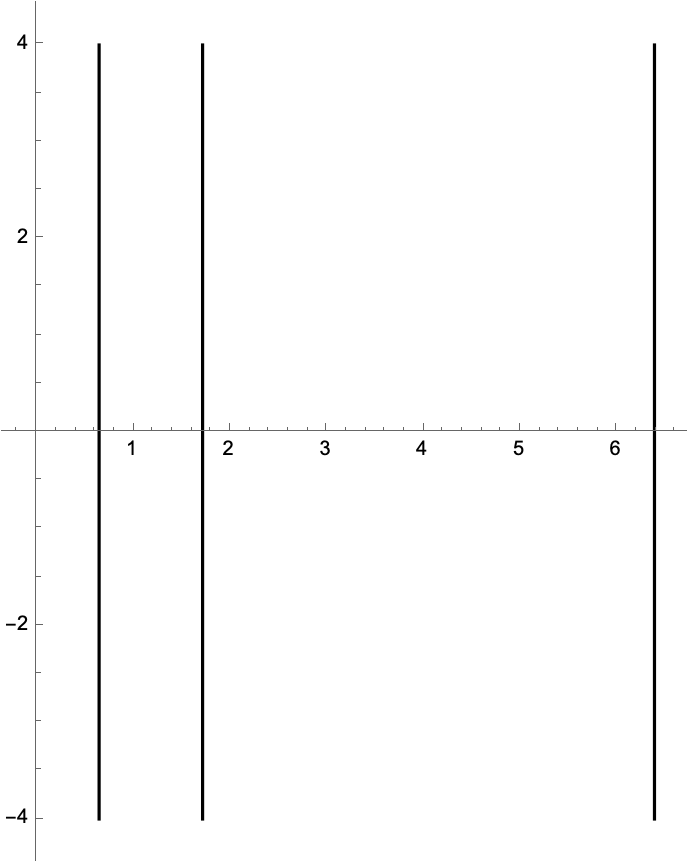}
 \caption{The lines $\psi_{t}(\{\Re s = 0\})$ for the maps $\psi_{t}(s) = e^{t} s + (e^t - 1)$. These lines drift off to infinity as $t \to \infty$.}
 \label{Figure_Lines}
 \end{center}
\end{figure}
Now consider the semigroup $\{V_t\}_{t \geq 0}$ of operators on $L^2(\R_{+})$ given by
$$(V_{t} g)(x) := e^{-t} e^{-(1 - e^{-t}) x} g(e^{-t} x), \quad t \geq 0.$$
Using the unitary operator
\begin{equation}\label{TTTTT}
T: L^2(\R) \to L^2(\R_{+}), \quad (T h)(x) := \frac{1}{\sqrt{x}} h(\log x), \quad x > 0,
\end{equation}
(for which $T^{-1} g(y) = e^{y/2} g(e^y)$), one sees that
$$(T^{-1} V_{t} T h)(y) = e^{-t/2} e^{-(1 - e^{-t}) e^{y}} h(y - t), \quad y \in \R.$$

Finally, with the weight function
\begin{equation}\label{weight}
w(y) := e^{-2 (e^y - 1)}
\end{equation}
(see Figure \ref{Figure_weight}), one can show that the operator
\begin{equation}\label{WWW}
W: L^2(\R) \to L^2(\R, w(y) dy), \quad (W h)(y) := \frac{h(y)}{\sqrt{w(y)}},
\end{equation}
 is unitary and  a calculation from \cite[Prop.~2.4]{GP} shows that
$$(W \sigma_{t} W^{-1} f)(y) = f(y - t), \quad f \in L^2(\R, w(y) dy), \quad t > 0.$$ This gives us the following result from \cite[Prop.~2.4]{GP}.

\begin{Proposition}\label{sijgfoishiftTTT}
The semigroup $\{\sigma_t\}_{t \geq 0}$ on $L^2(\R)$ given by
$$(\sigma_{t} h)(y) := e^{-(1 - e^{-t}) e^{y}} h(y - t), \quad y \in \R,$$
is unitarily equivalent, via
$$\mathfrak{F} := W T^{-1} \mathcal{L}^{-1} \mathcal{U}: H^2 \to L^2(\R, w(y) dy),$$ to the semigroup
$$(S_{t} f)(y) := f(y - t)$$ acting on $L^2(\R, w(y) dy)$.
\end{Proposition}

Putting this all together, yields the following result from \cite[Thm.~2.5]{GP}.

\begin{Theorem}
A closed subspace $\mathcal{M}$ of $H^2$ is invariant for the Ces\`{a}ro operator if and only if the closed subspace $\mathfrak{F} \mathcal{M}^{\perp}$ of $L^2(\R, w(y)dy)$ satisfies $S_{t} (\mathfrak{F} \M^{\perp}) \subset \mathfrak{F} \mathcal{M}^{\perp}$ for all $t \geq 0$.
\end{Theorem}

For the unweighted $L^2(\R)$ space, the closed subspaces $\mathcal{F}$ of $L^2(\R)$ for which $S_t \mathcal{F} \subset \mathcal{F}$ for all $t \geq 0$ are described by the Beurling--Lax theorem (see \cite[p.~204]{MR3890074}).
For a wide class of weights $v$ on $\R$, the same result is true for the $\{S_{t}\}_{t \geq 0}$-invariant subspaces of $L^2(\R, v(y) dy)$. However, for the weight $w$ from \eqref{weight}, there is a different lattice of $\{S_{t}\}_{t \geq 0}$- invariant subspaces. Some restrictions on these subspaces were known by Domar \cite[Eqn.~8]{MR0645094} and the paper \cite{GP} gives some specific examples. The complexity of the $\{S_{t}\}_{t \geq 0}$-invariant subspaces stems from the fact that
the weight $w$ is uniformly bounded above and below on every interval $(-\infty,a)$ but
decreases rapidly on $(a,\infty)$. Thus,
the space $L^2((-\infty,a), w(y)dy)$ can be seen as a renormed version of $L^2(-\infty,a)$ whereas $L^2((a,\infty),w(y)dy)$ is a much larger space than $L^2(a,\infty)$.
\begin{figure}
\begin{center}
 \includegraphics[width=.6\textwidth]{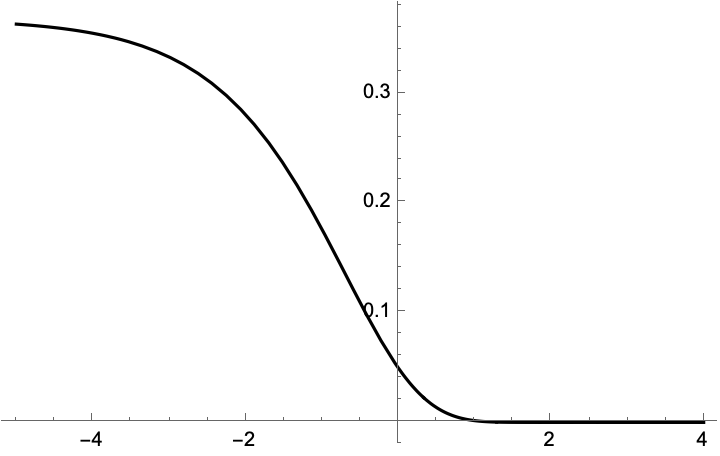}
 \caption{The weight function $w(y) = e^{-2 (e^y - 1)}$ from \eqref{weight}.}
 \label{Figure_weight}
 \end{center}
\end{figure}

\section{Finite co-dimensional invariant subspaces -- first proof}\label{firstysydydsysdf}

 An integral computation, using the fact that the adjoint $C^{*}$ of the Ces\`{a}ro operator $C$ on $H^2$ given by
\begin{equation}\label{97344987t389476345987}
 (C^{*} f)(z) = \frac{1}{1 - z} \int_{z}^{1} f(\xi) d \xi,
 \end{equation}
shows that the linear span of the functions
$$(1 - z)^{\mu} (\log(1 - z))^{j}, \quad 0 \leq j \leq k,$$
where $\Re \mu > -\tfrac{1}{2}$ and $ k \in  \N_{0} := \N \cup \{0\}$ is fixed, is a finite dimensional invariant  subspace for $C^{*}$. For example, the one dimensional space $$\operatorname{span} (1 - z)^{\mu}$$ is an eigenspace for $C^{*}$, with eigenvalue $(\mu+1)^{-1}$, while the spaces $$\operatorname{span}\{(1 - z)^{\mu}, (1 - z)^{\nu}\} \quad \mbox{and} \quad \operatorname{span}\{(1 - z)^{\mu},  (1 - z)^{\mu} \log(1 - z)\}$$ are two dimensional invariant subspaces for $C^{*}$. Theorem \ref{thm:fdcstar} below describes all of the finite dimensional $C^{*}$-invariant subspaces. 
There are several ways of proving this result, and we shall
use the semigroup techniques developed in \S \ref{two}. An alternate path to this result runs through  some techniques of Aleman \cite{MR1255271} (see Theorem \ref{thm:fincodviakt} below).

\begin{Theorem}\label{thm:fdcstar}
The finite dimensional $C^{*}$-invariant subspaces of $H^2$ are spans of finite unions of sets of the form
$$\{(1 - z)^{\mu} (\log(1 - z))^{j}, \quad 0 \leq j \leq k\},$$
where $\Re \mu > -\tfrac{1}{2}$ and $k \in \N_0$.
\end{Theorem}

\begin{proof}
Our starting point is the result in \cite[Thm.~2.9]{GP} which says that the finite dimensional
$\{S_{t}\}_{t \geq 0}$-invariant subspaces of
$L^2(\R, w(y) dy)$ (recall the translations $S_{t}$ from Proposition \ref{sijgfoishiftTTT}) are all spanned by finite unions of
sets of the form
$$\{e^{\lambda y},ye^{\lambda y}, \ldots y^n e^{\lambda y}\}$$
for some $n \in \N_0$ and $\Re \lambda > 0$.
We now translate this into information about invariant subspaces of the  Ces\`aro operator via our semigroup discussion from \S \ref{two}.

Under the unitary operator $W^{-1} : L^2(\R, w(y) dy) \to L^2(\R)$ from \eqref{WWW}, the
function
$y^k e^{\lambda y}$ is sent to $$y^k e^{\lambda y}e^{-(e^y-1)}.$$
Next, under the unitary operator $T: L^2(\R) \to L^2(\R_{+})$ from \eqref{TTTTT}, this last function is sent to
$$(\log x)^k x^{\lambda-1/2} e^{-(x-1)}.$$ We may multiply by a meaningless constant and
instead consider the function
$$(\log x)^k x^{\lambda-1/2} e^{-x}.$$

The next step is to use the normalized Laplace transform $\mathcal{L}:L^2(\R_{+}) \to H^2(\C_+)$ from \eqref{Laplace}.
Observe that
\[
\int_0^\infty e^{-sx} x^{\lambda-1/2} e^{-x} \, dx = (s+1)^{-\lambda-1/2}\Gamma(\lambda+\tfrac{1}{2})
\]
(where $\Gamma$ is the standard gamma function)
and, differentiating with respect to $\lambda$, we deduce that
$$
\int_0^\infty e^{-sx}(\log x)^k x^{\lambda-1/2} e^{-x} \, dx $$
is equal to
\begin{equation}\label{777777yYYY55}
(s+1)^{-\lambda-1/2}\Big\{a_0+a_1 \log (s+1) + \ldots + a_k (\log (s+1))^k\Big\},
\end{equation}
for constants $a_0,\ldots,a_k$ (depending on $\lambda$ but not $s$).

Finally, with the unitary operator $\mathcal{U}^{-1}: H^2(\C^{+}) \to H^2$ from \eqref{UUUinverse}, the functions
from \eqref{777777yYYY55} are mapped to
 functions of the form
 \[
 (1-z)^{\lambda-1/2}\Big\{b_0+b_1 \log  (1-z) + \ldots + b_k (\log  (1-z))^k\Big\}.
 \]
 Note that $\Re(\lambda-\tfrac{1}{2}) > -\tfrac{1}{2}$ and we pick up the (only) eigenvectors
 $(1-z)^\mu$ with $\Re \mu> -\tfrac{1}{2}$.

 Next, two-dimensional invariant subspaces are spanned either by $(1-z)^{\mu_1}$ and $(1-z)^{\mu_2}$
 with $\mu_1 \ne \mu_2$, or by
 $(1-z)^\mu$ and $(1-z)^\mu (\log (1-z))$. Since $1-z \mapsto 1-e^{-t}z -1+ e^{-t} = e^{-t}(1-z)$
 under the composition semigroup, it is easy to see that this is an invariant subspace. And so on, for higher dimensions.
\end{proof}

\begin{Corollary}
The finite co-dimensional closed invariant subspaces of the Ces\`{a}ro operator  on $H^2$ are orthogonal complements of  spans of finite unions of sets of the form $$\{(1 - z)^{\mu} (\log(1 - z))^{j}, \quad 0 \leq j \leq k\},$$
where $\Re \mu > -\tfrac{1}{2}$ and $k \in \N_0$.
\end{Corollary}

\section{The Kriete--Trutt transform}\label{sectsdfserKT}

A result of Kriete and Trutt \cite{MR281025}, used to prove that the Ces\`{a}ro operator on $H^2$ is subnormal, involves the following transform. Observe that
$$\Re\big( \frac{w}{1 - w}\big) > -\tfrac{1}{2} \quad \mbox{for all $w \in \D$} $$ and so the family of functions
\begin{equation}\label{888UHHJHJH6667}
q_w(z) := (1 - z)^{\frac{w}{1 - w}}, \quad w \in \D,
\end{equation}
 belong to $H^2$.
 Since
 $$q_{\frac{n}{n + 1}}(z) = (1 - z)^{n} \quad \mbox{for all $n \geq 0$},$$
 and the linear span of $(1 - z)^{n}, n \geq 0$, is dense in $H^2$, we see that
 \begin{equation}\label{ljdfhgsodlfhjbgdsa}
\overline{ \operatorname{span}}\{q_{w}: w \in \D\} = H^2.
 \end{equation}

The integral formula from \eqref{97344987t389476345987} verifies the eigenvalue identity
 $$C^{*} q_{w} = (1 - w) q_w \quad \mbox{for all $w \in \D$}.$$

For $f \in H^2$, define the function $Kf$ by
\begin{equation}\label{78uqhwebnfmnewhgrygt7h}
(K f)(w) := \langle f, q_{\overline{w}}\rangle_{H^2}, \quad w \in \D.
\end{equation}
Note that $K f$ is analytic on $\D$ and \eqref{ljdfhgsodlfhjbgdsa} enables us to create a Hilbert space $\mathcal{H}$ of analytic functions on $\D$, we will call the {\em Kriete--Trutt space}, as the range of $\mathcal{H}$, i.e.,
$$\mathcal{H} := \{Kf: f \in H^2\}.$$
The Hilbert space structure on $\mathcal{H}$ comes from the range norm
$$\|K f\|_{\mathcal{H}} := \|f\|_{H^2}.$$
With this norm, the {\em Kriete--Trutt transform} $K$ defined in  \eqref{78uqhwebnfmnewhgrygt7h} is a unitary map from $H^2$ onto $\mathcal{H}$.

 For $f \in H^2$ and $z \in \D$, use the eigenvalue identity from \eqref{97344987t389476345987} to see that
 \begin{align*}
(K C f)(z) & =
\langle C f, q_{\overline{z}}\rangle_{H^2}\\ & = \langle f, C^{*} q_{\overline{z}}\rangle_{H^2} \\
& = \langle f, (1 - \overline{z}) q_{\overline{z}}\rangle_{H^2}\\
& = (1 - z) \langle f, q_{\overline{z}}\rangle_{H^2}\\
& = (1 - z) (K f)(z).
\end{align*}
This shows that the multiplication operator $M_z f = z f$ on $\mathcal{H}$ is well defined, bounded, and $\|M_{z}\| = 1$ (since $\|I - C\| = 1$ from Proposition \ref{BrownShields}).
We summarize the discussion above with the following.

\begin{Proposition}[Kriete--Trutt \cite{MR281025}]\label{KTTTT}
The operator $K: H^2 \to \mathcal{H}$ is unitary and $K (I - C) K^{*} = M_{z}$, multiplication by $z$, on $\mathcal{H}$.
\end{Proposition}

A further analysis of Kriete and Trutt  from \cite{MR281025} shows that $\mathcal{H}$ contains the analytic polynomials $\C[z]$ as a dense set and, more importantly (and quite difficult to prove), there is a positive finite Borel measure $\mu$ on $\overline{\D}$ such that
\begin{equation}\label{qppoUHggfffTT}
\int_{\D} |p|^2 d \mu = \|p\|^{2}_{\mathcal{H}} \quad \mbox{for all $p \in \C[z]$}.
\end{equation}
This measure $\mu$ is supported on the sequence of circles
$$\gamma_{n} = \Big\{z: \Big|z - \frac{n}{n + 1}\Big| = \frac{1}{n + 1}\Big\} \quad \mbox{for $n  \geq 0$};$$
(see Figure \ref{KTC}).
\begin{figure}
\begin{center}
 \includegraphics[width=.6\textwidth]{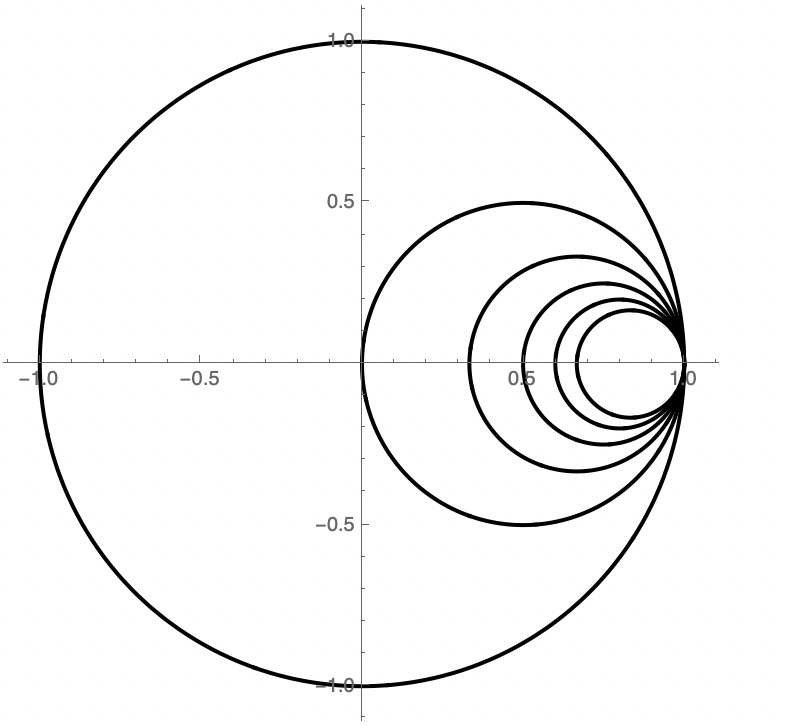}
 \caption{The circles (from left to right) $\gamma_0, \gamma_1, \gamma_2, \gamma_3, \gamma_4, \gamma_5$.}
 \label{KTC}
 \end{center}
\end{figure}
Furthermore, $\mu(\gamma_{n}) = 2^{-n - 1}$ and $\mu|_{\gamma_n}$ is mutually absolutely continuous with respect to arc length measure on $\gamma_{n}$.
 This says that $\mathcal{H} = \mathcal{H}^{2}(\mu)$ (the closure of $\C[z]$ in $L^2(\mu)$) and thus, $C$ is unitarily equivalent to $I - M_z$ on $\mathcal{H}^2(\mu)$. This last operator is subnormal (with $M_{1 - z}$ on $L^2(\mu)$ being a normal extension) and hence $C$ is subnormal.

A special case of a theorem of Aleman \cite{MR1255271} is the following theorem, which can
be proved by an analogue of an argument from \cite{MR933319} (see also \cite[Ch.~10]{MR4249569}).

\begin{Theorem}\label{thm:fincodviakt}
A closed subspace $\mathcal{M}$ of $\mathcal{H}^2(\mu)$ is a finite co-dimensional $M_z$-invariant subspace if and only if there is a polynomial $p$ all of whose zeros lie in $\D$ such that $\mathcal{M} = p \mathcal{H}^2(\mu)$.
\end{Theorem}

The following two propositions are relevant to such considerations (and for some examples in Section \ref{KTexamples}), so we give   short self-contained proofs.

\begin{Proposition}\label{COIADSASDSAD} 
The space $\mathcal{H}^2(\mu)$ has the so-called {\em division property}: if $\lambda \in \D$ and $f \in \mathcal{H}^2(\mu)$ with $f(\lambda) = 0$, then $f(z)/(z - \lambda) \in \mathcal{H}^2{\mu)}$. 
\end{Proposition}

\begin{proof}
The measure $\mu$ from $\mathcal{H}^2(\mu)$ is a finite positive Borel measure on $\overline{\D}$ with no point masses in $\D$ (recall the discussion at the end of \S \ref{sectsdfserKT}). Moreover, using the fact that $\mathcal{H}^2(\mu)$ is a reproducing kernel Hilbert space of analytic functions on $\D$ \cite[Lemma 3]{MR281025}, it follows that for each $\lambda \in \D$, there is a positive constant $C_{\lambda}$ such that
\begin{equation}\label{RKHS}
|f(\lambda)|^2 \leq C_{\lambda} \int |f|^2 d \mu \quad \mbox{for all $f \in \mathcal{H}^2(\mu)$}.
\end{equation}
Fix $\lambda \in \D$ and, for $f \in \mathcal{H}^2(\mu)$, consider the difference quotient
\begin{equation}\label{QQQQQ}
(Q_{\lambda} f)(z) := \frac{f(z) - f(\lambda)}{z - \lambda}, \quad z \in \D.
\end{equation}
Note that $Q_{\lambda} f$ is analytic on $\D$ (since $f$ is analytic on $\D$).

Let us first show that $Q_{\lambda} f \in L^2(\mu)$.
Fix $r > 0$ small enough so that $\overline{D(\lambda, r)} \subset \D$ (where $D(\lambda, r) = \{z: |z - \lambda| < r\}$).
Then
$$\int |Q_{\lambda} f|^2 d \mu = \int_{\overline{D(\lambda, r)}} |Q_{\lambda} f|^2 d \mu + \int_{\overline{\D} \setminus \overline{D(\lambda, r)}} |Q_{\lambda} f|^2 d \mu.$$
The first integral above is bounded above by
$$\mu(\overline{D(\lambda, r)}) \cdot \sup_{z \in D(\lambda, r)} \big|\frac{f(z) - f(\lambda)}{z - \lambda}\big|^2$$
which is finite.
The second integral is bounded above by
$$\frac{1}{r^2} \int_{\overline{\D}} |f - f(\lambda)|^2 d \mu$$
which is also finite.

Next we show that $Q_{\lambda} f \in \mathcal{H}^2(\mu)$. Let $\{p_n\}_{n \geq 1}$ be a sequence polynomials which approximate $f \in \mathcal{H}^{2}(\mu)$ in $L^2(\mu)$ norm.
From \eqref{RKHS} one can see that $p_n \to f$ pointwise in $\D$.
By \cite[Lemma 2]{MR350489}, there is a positive constant $C_{w}$  such that
\begin{equation}\label{bbbbel}
\|(z - \lambda) q\|_{\mathcal{H}^2(\mu)} \geq C_{w} \|q\|_{\mathcal{H}^2(\mu)} \quad \mbox{for all $q \in \C[z]$}.
\end{equation}
Apply this inequality  to the sequence of polynomials
$$q_{n}(z) = \frac{p_{n}(z) - p_{n}(\lambda)}{z - \lambda},$$
to see that $q_n \to f$ pointwise in $\D$ and are norm bounded. By \cite[Prop.~2]{MR0748841}, $q_n \to Q_{\lambda} f$ weakly in $\mathcal{H}^2(\mu)$. This means that $Q_{\lambda} f$ belongs to the weak closure of the polynomials. However, the weak closure and norm closure are the same, since the polynomials form
a convex set (this is sometimes known as Mazur's theorem \cite[Thm. V.3.13]{MR1009162}).
Moreover, by the closed graph theorem, the map $f \mapsto Q_{\lambda} f$ defines a bounded linear operator on $\mathcal{H}^2(\mu)$.
\end{proof}
See a discussion from \cite{MR0911086} for more on the division property for  general Banach spaces of analytic functions.

Finally, we make the following observation.

\begin{Proposition}\label{cyclicicicic}
The vector space $(z - \xi) \mathcal{H}^2(\mu)$ is dense in $\mathcal{H}^2(\mu)$ whenever $\xi \in \T$.
\end{Proposition}

\begin{proof}
By \cite[Theorem 4]{MR350489} the inclusion map $i: H^2 \to \mathcal{H}^2(\mu)$ is continuous. Thus, there is a positive constant $C$ such that
\begin{equation}\label{injectioncontinuous}
\|p(z) (z - \xi) - 1\|_{H^2} \geq C \|p(z) (z - \xi) - 1\|_{\mathcal{H}^{2}(\mu)} \quad \mbox{for all $p \in \C[z]$}.
\end{equation}
Since $(z - \xi)$ is an outer function, Beurling's theorem \cite[Thm.~7.4]{Duren} shows that the left hand side of the above can be made as small as desired by choosing an appropriate $p \in \C[z]$.

The previous paragraph shows that 
$1 \in \mathcal{S} := \overline{(z - \xi) \mathcal{H}^2(\mu)}$ and the $M_z$-invariance of $\mathcal{S}$ will show that $\mathcal{S} = \mathcal{H}^2(\mu)$.
\end{proof}

\section{Other interesting invariant subspaces}\label{KTexamples}

Although the finite co-dimensional $M_z$-invariant subspaces of the Kriete--Trutt space $\mathcal{H}^2(\mu)$ are thoroughly understood (Theorem \ref{thm:fincodviakt}), a general $M_z$-invariant subspace of $\mathcal{H}^2(\mu)$ can be quite complicated.   Recall the family of $H^2$ functions $q_w$, $w \in \D$, from \eqref{888UHHJHJH6667}.

\begin{Proposition}
If $A \subset \D$ is a zero sequence for $\mathcal{H}^{2}(\mu)$, meaning there is an $f \in \mathcal{H}^{2}(\mu) \setminus \{0\}$ which vanishes on $A$,  then
$$\overline{\operatorname{span}}\{q_{a}: a \in A\}$$ is a nontrivial invariant subspace of the adjoint of the Ces\`{a}ro operator.
\end{Proposition}

\begin{proof}
Notice that
\begin{align*}
(\overline{\operatorname{span}}\{q_{a}: a \in A\})^{\perp} & = \{f \in H^2: \langle f, q_a \rangle = 0\;  \forall a \in A\}\\
& = \{f \in H^2: (K f)(a) = 0 \; \forall a \in A\}\\
& = \{g \in \mathcal{H}^2(\mu): g|_{A} \equiv 0\} \not = 0.
\end{align*}
This shows that $\overline{\operatorname{span}}\{q_{a}: a \in A\} \not = H^2$. From the fact that
$$C^{*} q_a = (1 - a) q_a,$$
it follows that $\overline{\operatorname{span}}\{q_{a}: a \in A\}$ is $C^{*}$-invariant.
\end{proof}

\begin{Corollary}
If $A \subset \D$ is a zero sequence for $\mathcal{H}^2(\mu)$, then
$$(\overline{\operatorname{span}}\{q_{a}: a \in A\})^{\perp}$$ is a nontrivial invariant subspace of the Ces\`{a}ro operator.
\end{Corollary}

The zero sequences of $\mathcal{H}^2(\mu)$ can be quite complicated (see \S \ref{ramrekresas}) and do not have a complete characterization. 

One can use the discussion in the previous section (Proposition \ref{COIADSASDSAD} in particular) to create interesting ``chains'' of invariant subspaces for the Ces\`{a}ro operator.
Since the functions $(1-z)^\alpha$, $\Re \alpha> - \tfrac{1}{2}$, are eigenfunctions for $C^*$,
we may construct invariant subspaces by joining together eigenspaces. The following result \cite[Cor. 7]{MR350489} will be of use here.

\begin{Corollary}
Let $\{\lambda_n\}_{n \geq 1}$ be a sequence of positive numbers satisfying $\lambda_{n+1}-\lambda_n \geq\delta>0$ for all $n \in \N$ and consider the quantity
\[
b_r := \sum_{\lambda_n<r} \frac{1}{\lambda_n} - a \log r
\]
for $r>0$,
where $a >0$ is independent of $r$. If $b_r$ is unbounded above as $r \to \infty$ for some $a \geq \tfrac{1}{2}$
then
$$\operatorname{span}\{(1-z)^{\lambda_n}: n \geq 1\}$$ is a  dense subset of $H^2$. However, if  $b_r$ is bounded above
as $r \to \infty$ for some $a< \tfrac{1}{2}$ then
$$\operatorname{span}\{(1-z)^{\lambda_n}: n \geq 1\}$$ is not a dense subset of $H^2$.
\end{Corollary}

For example, for each $k \in \N$, the set
$$\Lambda_k := \{k+3n: n \geq 1 \}$$ yields a proper closed $C^*$-invariant subspace $V_k$ of $H^2$ defined by
$$V_k := \overline{\operatorname{span}} \{(1-z)^{\lambda_n}: n \geq 1\}.$$
(take $\tfrac{1}{3} < a < \tfrac{1}{2}$ in the above corollary).
 However, the join (the closure of $V_{k} + V_{\ell}$) of any two such subspaces
$V_k$ and $V_\ell$
with $k \not\equiv \ell$ mod 3 will be all 
of $H^2$ (take $\tfrac{1}{2} < a < \tfrac{2}{3}$).

We can actually say a bit more about these $C^{*}$-invariant subspaces $V_k$. Translating this into the Kriete--Trutt space $\mathcal{H}^2(\mu)$ determined by $f \mapsto K f$,
upon which $C$ is unitarily equivalent to $M_{1 - z}$ (recall Proposition \ref{KTTTT}),
the images (under $K$) of the orthogonal complements of the subspaces $V_k$ above (which are clearly invariant under $M_{1 - z}$) are subspaces of $\mathcal{H}^2(\mu)$ consisting of
functions that vanish at  $w=j/(j+1)$, $j \geq 1$.

Now the division property for $\mathcal{H}^2(\mu)$ (recall that the difference quotient $Q_{\lambda} f$ belongs to $\mathcal{H}^2(\mu)$ whenever $f \in \mathcal{H}^2(\mu)$  and $\lambda \in \D$, from \eqref{QQQQQ})  can be used to show that, for example,
$$V_1 \supsetneq V_4 \supsetneq V_7 \supsetneq \ldots$$
(or equivalently, that $(1-z) \not\in V_4$, etc.). Indeed, by  the division property, it is possible to construct
functions in $\mathcal{H}^2(\mu)$ that vanish at $j/(j+1)$ for $j=4,7,10,\ldots$ but not for $j=1$.

\section{The role of model spaces}\label{modelrole}

This section contains a very curious class of invariant subspaces for the Ces\`{a}ro operator coming from the well studied theory of model spaces. For $\alpha > 0$ let $u_{\alpha}$ be the atomic inner function
\begin{equation}\label{bkkCKCIi}
u_{\alpha}(z) := \exp\Big(\alpha \frac{z + 1}{z - 1}\Big), \quad z \in \D.
\end{equation}
Notice that $u_{\alpha}$ is a bounded analytic function on $\D$ and $|u_{\alpha}(e^{i \theta})| = 1$ for all $0 < \theta < 2 \pi$. Thus, $u_{\alpha} H^2$ is a closed subspace of $H^2$.
From here one can consider the {\em model  space}
$$(u_{\alpha} H^2)^{\perp}  :=  \{f \in H^2: f \perp u_{\alpha} H^2\}.$$ We refer the reader to \cite{MR3526203} for the basics of model spaces.

Let us denote $H^{\infty}(\D)$ to be the set of bounded analytic functions on $\D$ and for $g \in H^{\infty}(\D)$, define 
$$\|g\|_{\infty} := \sup_{z \in \D} |g(z)|.$$

As noted by Professor Aleman, the following result can also be derived from a similar result on 
the Volterra operator, which can be found in \cite{MR2445578}.  We give a direct proof using semi-groups since they will be used to obtain further results below. 

\begin{Theorem}\label{thm:ualphinv}
For each $\alpha > 0$, the model space $(u_{\alpha} H^2)^{\perp}$ is a closed invariant subspace of the Ces\`{a}ro operator.
\end{Theorem}

\begin{proof}
By Theorem \ref{wydfgviuc8d7s6ed7sfv6} it suffices to show that
$$C_{\phi_t} (u_{\alpha} H^2) \subset u_{\alpha} H^2 \quad \mbox{for all $t  > 0$}.$$
 This comes from a special case of a result of Cowen and Wahl \cite[Lemma 5]{MR3266985}. For the sake of completeness, we outline the proof here. Observe that
\begin{equation}\label{66TTytr445433}
\Re\Big(\frac{\phi_t(z) + 1}{\phi_{t}(z)- 1} - \frac{z + 1}{z - 1}\Big) \leq 0 \quad \mbox{for all $ z \in \D$}.
\end{equation}
This follows from the simple fact that $\phi_t(\T)$ is a circle inside $\overline{\D}$ that is internally tangent to $\T$ at $\{1\}$ and that the linear fractional transformation
$$z \mapsto \frac{z + 1}{z - 1}$$ maps $\D$ onto the left half plane $-\C^{+}$ (One can also see \eqref{66TTytr445433} via Julia's Lemma).
For any $g \in H^2$ we have
\begin{align}
& C_{\phi_t} (u_{\alpha}(z) g(z)) = \exp\Big(\alpha \frac{\phi_t(z) + 1}{\phi_t(z) - 1}\Big) g (\phi_t(z))\nonumber\\
& = \Bigg(\exp\Big(\alpha \frac{z + 1}{z - 1}\Big) \exp\Big(\alpha\Big( \frac{\phi_t(z) + 1}{\phi_{t}(z)- 1} - \frac{z + 1}{z - 1}\Big)\Big)\Bigg) g (\phi_t(z))\nonumber\\
& = u_{\alpha}(z) \exp\Big(\alpha \Big(\frac{\phi_t(z) + 1}{\phi_{t}(z)- 1} - \frac{z + 1}{z - 1}\Big)\Big) g (\phi_t(z)).\label{ooocCC}
\end{align}
From \eqref{66TTytr445433}  it follows that the function
$$\exp\Big(\alpha \Big(\frac{\phi_t(z) + 1}{\phi_{t}(z)- 1} - \frac{z + 1}{z - 1}\Big)\Big) $$
belongs to $H^{\infty}(\D)$. Since $g(\phi_t) \in H^2$, \eqref{ooocCC} shows that
$$C_{\phi_t} (u_{\alpha} g) \in H^2 \quad \mbox{for all $t \geq 0$},$$
which completes the proof.
\end{proof}

As it turns out (see Theorem \ref{thm:onlymodel} below and a discussion of the Volterra operator in \cite{MR2445578}), the family of model spaces
$$\{(u_{\alpha} H^2)^{\perp}: \alpha > 0\}$$ are the {\em only} model spaces that are invariant under the Ces\`{a}ro operator. To prove this, we need a few reminders about inner functions.

Recall \cite[Ch.~2]{Duren}  that an analytic function $u$ on $\D$ is {\em inner} if $u \in H^{\infty}(\D)$ and the radial boundary function
$$u(\xi) := \lim_{r \to 1^{-}} u(r \xi),$$
which exists for almost every $\xi \in\T$ by Fatou's theorem \cite[Thm.~1.3]{Duren}, satisfies $|u(\xi)| = 1$ for almost every $\xi \in \T$. Any inner function can be factored (uniquely up to a unimodular constant factor) as $u = B s_{\mu}$, where $B$ is the Blaschke product
$$B(z)  := z^{N} \prod_{n = 1}^{\infty} \frac{\overline{z_n}}{|z_n|} \frac{z_n - z}{1 - \overline{z_n} z}, \quad z \in \D, $$
with zeros at $z = 0$ and at $z_n \in \D \setminus \{0\}$ satisfying the Blaschke condition $\sum_{n \geq 1} (1 - |z_n|) < \infty$, and $s_{\mu}$ is the singular inner function
$$s_{\mu}(z) := \exp\Big(-\int_{\T} \frac{\xi + z}{\xi - z} d \mu(\xi)\Big), \quad z \in \D,$$
with associated singular measure (with respect to Lebesgue measure $m$) on $\T$. Note that $u_{\alpha} = s_{\alpha \delta_{1}}$, where $\alpha > 0$. Both $B$ and $s_{\mu}$ are inner. An important tool needed below is {\em boundary spectrum}
\begin{equation}\label{spectrum}
\Sigma(u) := \Big\{\xi \in \T: \liminf_{z \in \D, z \to \xi} |u(z)| = 0\Big\}
\end{equation}
of $u$. It is known (see \cite[Thm.~7.18 and Prop.~7.19]{MR3526203}) that $\Sigma(u)$ is a closed subset of $\T$, $u$ has an analytic continuation across $\T \setminus \Sigma(u)$, and $\Sigma(u)$ consists of the accumulation points of the zeros of $B$ and the support of the singular measure $\mu$.

There is a corresponding notion of inner function for $H^2(\C^+)$. A bounded analytic function $\Theta$ on $\C^{+}$ is inner if $|\Theta(i y)| = 1$ for almost every $y \in \R$. Note that $\Theta$ is inner for $\C^{+}$ if and only if $u = \Theta \circ \gamma$ is  inner for $\D$, where $\gamma$ is the conformal map from $\D$ to $\C^{+}$ from \eqref{gammaC}.  Though there is an analogous factorization theorem for inner functions $\Theta$ on $\C^{+}$ as there was for inner functions $u$ on $\D$ (see \cite[Ch.~2]{Garnett} or \cite[Ch.~6]{MR3890074}), we will only need the fact that the function
\begin{equation}\label{hghhghHJHJHJH}
\Theta(s) = e^{-a s}, \quad s \in \C^{+},
\end{equation}
where $a > 0$, is inner on $\C^+$.

As we did for the atomic inner functions $u_{\alpha}$, we can define the {\em model space} $(u H^2)^{\perp}$ for any inner function $u$.
As noted earlier, the following result can be derived from a similar result on the
Volterra operator on $H^2$  given in \cite{MR2445578}.

\begin{Theorem}\label{thm:onlymodel}
If $u$ is a nonconstant  inner function and $(u H^2)^{\perp}$ is an invariant subspace for the Ces\`{a}ro operator, then $u = u_{\alpha}$ for some $\alpha > 0$.
\end{Theorem}

\begin{proof}
We use Theorem \ref{wydfgviuc8d7s6ed7sfv6}. For the closed subspace $uH^2$ to be invariant under each composition operator $C_{\phi_t}$, $t \geq 0$, it
is necessary for $u$ to divide $u \circ \phi_t$ for each $t$ (meaning that $(u \circ \phi_t)/u \in H^{\infty}(\D)$). Thus, if $u$ has a zero $z_0 \in \D$ then $u(\phi_t(z_0))=0$ for every $t\geq 0$. As $t \to 0^{+}$, these zeros accumulate at $z = 0$
which is
impossible by the isolated zeros theorem for analytic functions.

By the classical factorization theorem for inner functions mentioned earlier (and the argument in the previous paragraph), we may now suppose that $u$ is a singular inner function $s_{\mu}$ and consider what happens
if a point $\xi_0 \in \T \setminus \{1\}$ lies in its boundary spectrum $\Sigma(u)$ from \eqref{spectrum}. Let us examine what happens to $C_{\phi_t} (u H^2)$ when $t=-\log 2$.
Assuming that $u(\phi_{-\log 2}) \in u H^2$, we see that  $u(z)$ divides the bounded analytic  function $u((1+z)/2)$, say
$u((1+z)/2)=u(z) g(z)$ where $g\in H^\infty(\D)$ (removing inner factors does not change the $H^\infty$ norm).

Now observe that
\begin{equation}\label{9w87ryugoerfwd}
 \liminf_{z \in \D, z \to \xi_0} |u((1+z)/2)| = |u((1+\xi_0)/2)| \ne 0
 \end{equation}
(since $u$ is a singular inner function and thus has no zeros in $\D$). However, by the definition of boundary spectrum from \eqref{spectrum},
\begin{equation}\label{nbbVVVGHBN}
\liminf_{z \in \D, z \to \xi_0} |u(z)|= 0.
\end{equation}
Since $u((1+z)/2)=u(z) g(z)$, the facts from \eqref{9w87ryugoerfwd} and \eqref{nbbVVVGHBN} contradict the fact that $g$ must be bounded near $\xi_0$.
\end{proof}

From Theorem \ref{thm:ualphinv},
the Ces\`{a}ro operator has the model spaces $(u_{\alpha} H^2)^{\perp}$ as  invariant subspaces for $\alpha>0$ (and no other model spaces are invariant under $C$).
Thus, its adjoint $C^*$ has the invariant subspaces $u_{\alpha} H^2$ and no other Beurling subspaces $u H^2$ are invariant. 

With the notation of \cite{GP}, we can see that these Beurling spaces $u_{\alpha} H^2$ are equivalent to the standard $\{S_{t}\}_{t \geq 0}$-invariant subspaces
$$L^2((a,\infty),w(y) dy) \quad \mbox{ with $a=\log \alpha$}.$$
For under \[
TW^{-1}: L^2(\R, w(y) dy) \to L^2(0,\infty)
\] (recall the operators $T$ and $W$ from \S \ref{two}) such a subspace is sent to $L^2(e^a,\infty)$. Now, under the  normalized Laplace transform $\mathcal{L}$, it
maps to the space $\Theta_a H^2(\C^+)$, where $\Theta_a$ is the inner function
$\Theta_{a}(s) = \exp(-e^a s)$ from \eqref{hghhghHJHJHJH}.
Finally, under  the unitary operator $\mathcal{U}^{-1}: H^2(\C^+) \to H^2$ from \eqref{UUUinverse} we arrive at $u_\alpha H^2$ with $\alpha=e^a$.

The model spaces $(u_{\alpha} H^2)^{\perp}$, $\alpha > 0$,  also have the following interesting property.

\begin{Theorem}\label{indewejrhewf777}
If $\mathcal{M}$ is a closed invariant subspace for the Ces\`{a}ro operator such that $(u_{\alpha} H^2)^{\perp} \subset \mathcal{M}$ for some $\alpha > 0$, then $\mathcal{M} = (u_{\beta} H^2)^{\perp}$ for some $\beta \geq \alpha$.
\end{Theorem}

\begin{proof}
As we have just seen, the transformation $\mathfrak{F} = W T^{-1} \mathcal{L}^{-1} \mathcal{U}$ maps the subspace $u_\alpha H^2$ to
$L^2((a,\infty),w(y) dy)$, with $a=\log\alpha$. This has  the remarkable property that the only
translation-invariant subspaces contained in it are standard, that is, of the form
$L^2((b,\infty),w(y) dy)$, for $b \geq a$
 (this follows from
Domar's work and appears in the discussion preceding Theorem 2.9 in \cite{GP}).
On applying $\mathfrak{F}^{-1}$ the result follows.
\end{proof}

We now examine the $C$-invariant subspaces which are {\em contained} in a fixed model space  $(u_{\alpha} H^2)^{\perp}$.

\begin{Theorem}\label{zbzbssjjquUU}
Fix $\alpha > 0$. Up to isomorphism, the restriction of $I-C$ to the model space $(u_{\alpha} H^2)^{\perp}$ is a shift of multiplicity one and its
invariant subspaces are parameterized by Beurling spaces of the form $\Theta H^2(\C^+)$
where $\Theta \in H^\infty(\C^+)$ is inner.
\end{Theorem}

\begin{proof}
To examine the $C$-invariant subspaces contained in a model space $(u_{\alpha} H^2)^{\perp}$, let us return to the translation model on $L^2(\R, w(y)dy)$ from \S \ref{two}
and note that the adjoint of the weighted translation $e^{t/2}S_t$ is given by
\[
(T_t g)(y)=e^{t/2} g(y+t) \frac{w(y+t)}{w(y)}, \quad y \in \R.
\]
We seek the common invariant subspaces for $\{T_t\}_{t \geq 0}$ that are contained in the space $L^2((-\infty,a), w(y) \, dy)$.

Note that the weight $w$ is uniformly bounded above and below on the interval $(-\infty,a)$ (see Figure \ref{Figure_weight}), and
thus the operator
$$V: L^2((-\infty,a), w(y) \, dy) \to L^2(-\infty,a), \quad (Vg)(y): =\frac{g(y)}{w(y)}$$ is an isomorphism (but not an isometry).

The question therefore reduces to parameterizing the common invariant subspaces
of the semigroup $\{V T_t V^{-1}\}_{t \geq 0}$ on $L^2(-\infty,a)$, and we see
that
\[
(V T_t V^{-1} g)(y)=e^{t/2} g(y+t),
\]
is a backward translation.
Now we can use the Beurling--Lax theorem to find the lattice of common invariant
subspaces, but there is an alternative method that gives us more
information about the action of $C$ on $(u_\alpha H^2)^\perp$.

Namely, we use a reversal operator
$$R: L^2(-\infty,a) \to L^2(0,\infty), \quad (Rf)(y)=f(a-y)$$ and then, applying the normalized Laplace transform $\mathcal{L}$,  we have
the isomorphism
$$ \mathcal{L} R: L^2(-\infty,a) \to H^2(\C^+).$$
Under this, the weighted left translation $e^{t/2} S_t$ becomes
\[
(\mathcal{L} R)VT_t V^{-1}(\mathcal{L} R)^{-1} G(s)= e^{t/2} e^{-st} G(s), \quad s \in \C^+.
\]
Let $A$ be the infinitesimal generator of the composition
semigroup $\{C_{\phi_t}\}_{t \geq 0}$, namely
\[
(Af)(z)=(1-z)f'(z), \quad z\in \mathbb{D},
\]
(see the pioneering work by Berkson and Porta \cite{MR480965}, for instance). Then
we may  represent the resolvent $(I-A^*)^{-1}$
of the adjoint   semigroup -- which is the restriction of
the Ces\`aro operator -- equivalently as multiplication by the function\[
F(s)= \int_0^\infty e^{-t} e^{t/2} e^{-st} \, dt = \frac{1}{s+\tfrac{1}{2}}.
\]
It is more productive now to use $I-C$  in which case
we have the operator of multiplication by the function
\[
\dfrac{s-\tfrac{1}{2}}{s+\tfrac{1}{2}}.
\]
That is,
up to isomorphism, the restriction of $I-C$ is a shift of multiplicity one and its
invariant subspaces are parameterized by Beurling spaces of the form $\Theta H^2(\C^+)$
where $\Theta \in H^\infty(\C^+)$ is inner.
\end{proof}

From Proposition \ref{BrownShields} we know that
$$\sigma(I - C) = \overline{\D}.$$
The construction above yields the following.   

\begin{Corollary}
For $\alpha > 0$, the spectrum of $I-C$ restricted to the model space $(u_{\alpha} H^2)^{\perp}$, or any of its invariant subspaces, is $\overline{\D}$.
\end{Corollary}

\begin{proof}
From the proof of Theorem \ref{zbzbssjjquUU}, the operator $I - C$, when restricted to a closed $C$-invariant subspace of a model space $(u_\alpha H^2)^\perp$
 will be the same as the spectrum of the operator
 of multiplication by
 $$b(s)=\frac{s - \frac{1}{2}}{s + \frac{1}{2}}$$ on $\Theta H^2(\C^+)$
 for an appropriate inner function $\Theta$ on $\C^{+}$.
 This is independent of $\Theta$ as we can solve
 $$(b-\lambda)\Theta f=\Theta g$$
 for $f,g \in H^2(\C^+)$ precisely when we can solve
 $(b-\lambda)  f=  g$. The spectrum of the analytic Toeplitz operator
 $M_b$ on $H^2(\C^+)$ is easily seen to be $\overline{b(\C^+)}=\overline\D$  (this is
 a special case of what is sometimes known
 as Wintner's theorem \cite[p.~365]{MR4545809}). 
\end{proof}

We see also that
$$\operatorname{dim} (\mathcal{M}/(I - C) \mathcal{M})= 1$$
for all $C$-invariant subspaces $\mathcal{M} \subset (u_{\alpha} H^2)^{\perp}$,
since
$$\operatorname{dim} \Big((u_{\alpha} H^2)^{\perp}/(I - C) (u_{\alpha} H^2)^{\perp}\Big) =
\operatorname{dim} (\Theta H^2(\C^{+})/b \Theta H^2(\C^{+})),$$ 
 and this last quantity is equal to one. Indeed, it can be shown, more generally, that every invariant
 subspace for $I-C$ has index 1, using the main theorem of
 Aleman, Richter and Sundberg \cite{MR2480609}. We discuss this further in Section
 \ref{ramrekresas}.

\begin{Remark}
Suppose that $\M$ is a closed $C$-invariant subspace of $H^2$. For fixed $\alpha > 0$ consider
$$\overline{\M + (u_{\alpha} H^2)^{\perp}}.$$
Notice how this is a $C$-invariant subspace that contains $(u_{\alpha} H^2)^{\perp}$. Thus, by Theorem \ref{indewejrhewf777}, either
$$\overline{\M + (u_{\alpha} H^2)^{\perp}} = (u_{\beta} H^2)^{\perp}$$ for some $\beta \geq \alpha$ or
$$\overline{\M + (u_{\alpha} H^2)^{\perp}} = H^2.$$
In the first case we have $\M \subset (u_{\beta} H^2)^{\perp}$ and, by Theorem \ref{zbzbssjjquUU}, we can, in a sense, describe $\M$.
Notice how the second possibility above says that $\M^{\perp} \cap u_{\alpha} H^2= \{0\}$.
An an example of when this occurs, one can consider a finite co-dimensional invariant subspace $\M$ consisting of eigenfunctions of $C$. Then $\M^{\perp} \cap u_{\alpha} H^2 = \{0\}$ and so $\overline{\M + (u_{\alpha} H^2)^{\perp}} = H^2$.
\end{Remark}

\section{Model spaces and the Kriete--Trutt space}

Results from the previous section bring up further interesting things one can say about the lattice of invariant subspaces for $M_z$ on the Kriete--Trutt space $\mathcal{H}^2(\mu)$. The Kriete--Trutt transform
$$K: H^2 \to \mathcal{H}^2(\mu), \quad (K f)(z) = \langle f, q_{\overline{w}}\rangle_{H^2},$$
where
$$q_w(z) = (1 - z)^{\frac{w}{1 - w}}$$
from \eqref{78uqhwebnfmnewhgrygt7h}
is a unitary operator which maps $C$-invariant subspaces of $H^2$ to $M_{z}$-invariant subspaces of $\mathcal{H}^2(\mu)$. What type of $M_z$-invariant subspace is $K ((u_{\alpha} H^2)^{\perp})$?

To begin to unpack this, we require a few further details about $\mathcal{H}^2(\mu)$ from \cite{MR350489}. The first \cite[Thm.~4]{MR350489} (and mentioned earlier in \S \ref{seconduuus}) is that the inclusion map
$$i: H^2 \to \mathcal{H}^2(\mu), \quad i(f) = f,$$
is bounded. The second \cite[Lemma~4]{MR350489}  is that if $u$ is an inner function whose boundary spectrum $\Sigma(u)$ from \eqref{spectrum} does not contain the point  $1$, then the multiplication operator
$f \mapsto u f$ is bounded below on $\mathcal{H}^2(\mu)$ and so $u \mathcal{H}^2(\mu)$ is a closed subspace of $\mathcal{H}^2(\mu)$. Third \cite[Cor.~1]{MR350489}, for inner functions $A$ and $B$,
$$\overline{A \mathcal{H}^2(\mu)} = \overline{B \mathcal{H}^2(\mu)} \quad \mbox{if and only if $u_s A = c u_t B$}$$
for some $|c| = 1$ and some $s, t \geq 0$, where $u_t$ is the standard atomic inner function from \eqref{bkkCKCIi}. Thus,
an inner function $B$ is cyclic for $M_z$ on $\mathcal{H}^2(\mu)$ if and only if $B = c u_t$. They also note that
$$K \frac{1}{1 - (1 - a) z} = e^{t} u_{t}, \quad t = - (\log a)/2, \quad 0 < a < 1.$$
This shows that 
$$\frac{1}{1 - (1 - a) z} $$
 is a {\em cyclic vector} for $C$, meaning
 $$\overline{\operatorname{span}}\Big\{C^n \frac{1}{1 - (1- a) z}: n \geq 0\Big\} = H^2.$$
 Fourth \cite[Thm.~5]{MR350489}, for a closed  $M_{z}$-invariant subspace $\mathcal{M}$ of $\mathcal{H}^2(\mu)$, we have that
$\mathcal{M} = \overline{u \mathcal{H}^{2}(\mu)}$ for some inner function $u$ if and only if $\mathcal{M} \cap i(H^2)$ is dense in $\mathcal{M}$.

 One can describe $K((u_{\alpha} H^2)^{\perp})$ as a ``liminf'' space as follows. A theorem of Tumarkin \cite[Thm.~4.3.1]{Douglas} says that for a given $\alpha > 0$,  there is a sequence $\{B_n\}_{n \geq 1}$ of finite Blaschke products with simple zeros such that
$$(u_{\alpha} H^2)^{\perp}= \varliminf (B_n H^2)^{\perp},$$
meaning that given any $f \in (u_{\alpha} H^2)^{\perp}$ there are $f_n \in (B_n H^2)^{\perp}$ such that $f_n \to f$ in $H^2$ norm. Note that
$$(B_n H^2)^{\perp} = \operatorname{span} \Big\{\frac{1}{1 - \overline{\lambda} z}: \lambda \in B_n^{-1}(\{0\})\Big\}$$
is a convenient space of rational functions \cite[Prop.~5.6]{MR3526203}.

If $\lambda \in \D$ and $k_{\lambda}(z) = (1 - \overline{\lambda} z)^{-1}$, then \cite[p.~203]{MR350489}
$$(K k_{\lambda})(z) = (1 - \lambda)^{\frac{z}{1 - z}}.$$
When $0 < \lambda < 1$,
\begin{equation}\label{Caychcyexp}
(K k_{\lambda})(z) = c_{\lambda} \exp\Big(-a_{\lambda} \frac{1 + z}{1 - z}\Big)
\end{equation}
for some positive constants $c_{\lambda}$ and $a_{\lambda}$. As mentioned earlier, $k_{\lambda}$ is cyclic for $C$ and so $Kk_{\lambda}$ is cyclic for $M_{z}$ on $\mathcal{H}^2(\mu)$.

The discussion above says that
$$K((u_{\alpha} H^2)^{\perp}) = \varliminf \Big\{ \operatorname{span}\big\{(1 - \lambda)^{\frac{z}{1 - z}}: \lambda \in B_n^{-1}(\{0\})\big\}\Big\},$$
meaning that given any $g \in K((u_{\alpha} H^2)^{\perp})$, there are
$$g_n \in \operatorname{span}\big\{(1 - \lambda)^{\frac{z}{1 - z}}: \lambda \in B_n^{-1}(\{0\})\big\}$$
with $g_n \to g$ in $\mathcal{H}^2(\mu)$.
But what does $K((u_{\alpha} H^2)^{\perp})$ really contain? Is
$$K((u_{\alpha} H^2)^{\perp}) = \overline{u \mathcal{H}^2(\mu)}$$ for some inner function $u$?
 The answer is no .

 \begin{Proposition}\label{nosthnehrenaasdd}
 For any $\alpha > 0$, the $M_z$-invariant subspace $K((u_{\alpha} H^2)^{\perp})$ of $\mathcal{H}^{2}(\mu)$ is not equal to $\overline{u \mathcal{H}^2(\mu)}$ for any inner function $u$.
 \end{Proposition}

\begin{proof}

First note that $\overline{u H^2} = \overline{u \mathcal{H}^2(\mu)}$, since $H^2$ is a dense subspace
of $\mathcal{H}^2(\mu)$ and multiplication by $u$ is a bounded operator on both $H^2$ and $\mathcal{H}^2(\mu)$.

Certainly
$$K((u_{\alpha} H^2)^{\perp}) \not = \overline{u_t H^2}$$
for any $t > 0$
since, by  \eqref{Caychcyexp}, $u_t$ corresponds to a Cauchy kernel via $K$ and these are cyclic vectors for $C$.

Is $K((u_{\alpha} H^2)^{\perp}) = \overline{uH^2}$ where the inner function $u$ has a Blaschke factor? No, since otherwise, $u(\lambda_0) = 0$ for some $\lambda_0 \in \D$ and thus
$$0 = (K f)(\lambda_0) = \langle f(w), (1 - w)^{\frac{\lambda_0}{1 - \lambda_0}}\rangle_{H^2} \quad \mbox{for all $f \in (u_{\alpha} H^2)^{\perp}$}.$$
This would mean that
$$(1 - w)^{\frac{\lambda_0}{1 - \lambda_0}} \in u_{\alpha} H^2.$$
But this cannot be since $(1 - w)^{\lambda_0/(1 - \lambda_0)}$ is outer (Beurling's theorem).

Thus,  if $K((u_{\alpha} H^2)^{\perp}) = \overline{u H^2}=\overline{u\mathcal{H}^2(\mu)}$, then the inner function $u$ must be a singular inner function $s_{\nu}$.
Now, unless the measure $\nu$
consists of a single atom,   $u$ has nonconstant inner factors $J_1$ and $J_2$  with associated singular measures with disjoint supports.
In this case, it follows that the $C$-invariant subspaces containing $\overline{uH^2}$
are not totally ordered by inclusion, and so, by Theorem \ref{thm:onlymodel}, we cannot have $K((u_{\alpha} H^2)^{\perp}) = \overline{u H^2}$.

Finally, suppose that $K((u_{\alpha} H^2)^{\perp}) = \overline{u H^2} = u \mathcal{H}^2(\mu)$ where $u$ is a singular
inner function corresponding to a point mass at $p \in \T \setminus \{1\}$. From our discussion above, note that $u \mathcal{H}^2(\mu)$ is closed in $\mathcal{H}^2(\mu)$.
Then there is a function $g \in (u_{\alpha} H^2)^{\perp}$ such that $K g = u$. Now, and this is interesting on its own, if
$f \in H^2$, then
\[
\langle S^* f,  (1 - w)^{\frac{\bar z}{1 - \bar z}}\rangle_{H^2}
= \langle f, (w-1+1)  (1 - w)^{\frac{\bar z}{1 - \bar z}}\rangle_{H^2}.
\]
So that   $ K(S^*f  - f)(z)$ is $- Kf (\phi(z))$
where
\[
\frac{\phi(z)}{1-\phi(z)}=  \frac{z}{1-z}+1 ;
\]
 that is, $\phi(z)=1/(2-z)$. But now $S^*g-g$ belongs to $(u_{\alpha} H^2)^{\perp}$
and so $K(S^* g- g) \in u \mathcal{H}^2(\mu)$. However,  $- Kg (\phi(z))$ is analytic in
an open neighborhood of $p$, which is a contradiction unless $S^*g-g=0$, which is impossible.
\end{proof}

\begin{Remark}
The $M_z$-invariant subspaces $\mathcal{M}_{\alpha} := K((u_{\alpha} H^2)^{\perp})$ are not the ``standard'' Beurling-type $M_z$-invariant subspaces $\overline{u \mathcal{H}^{2}(\mu)}$. Moreover,  via Theorem \ref{indewejrhewf777}, they also inherit the interesting property that if $\mathcal{M}$ is an $M_z$-invariant subspace that contains $\mathcal{M}_{\alpha}$ for some $\alpha$, then $\mathcal{M} = \mathcal{M}_{\beta}$ for some $\beta \geq \alpha$.
\end{Remark}

\section{Cyclic subspaces}
\label{sec:cylicindex}

Recall that $I-C$   is unitarily equivalent to multiplication by the function $z$ on
the Kriete--Trutt space $\mathcal{H}^2(\mu)$ and is bounded below (recall \eqref{bbbbel}).
Thus, we may discuss the {\em index}
$$\operatorname{dim}(\mathcal{M}/ z \mathcal{M})$$
of an invariant
subspace $\mathcal{M}$. See \cite{MR0911086} for some general facts about the index of an invariant subspace of $M_z$ on Banach spaces of analytic functions.
We  start with some preliminary observations
about $C$ itself, which is not bounded below.

\begin{Proposition}\label{prop:apr28a}
Let $\mathcal{M}$ be a closed subspace of $\mathcal{H}^{2}(\mu)$ that is  invariant under $M_{1-z}$, and hence $M_z$. Then
$(1-z)\mathcal{M}$ is dense in $\mathcal{M}$.
\end{Proposition}

\begin{proof}

Let
$$p_n(z)=1-(z+z^2+...+z^{n})/n.$$
Since $p_{n}(1) = 0$, we have that $p_n(z)= (1-z) q_n(z)$ for some $q_n \in \C[z]$.

By Parseval's theorem, observe that $\|p_n h - h \|_{H^2} \to 0$ for every monomial $h(z)=z^k$ (hence, by \eqref{injectioncontinuous}, in $\mathcal{H}^2(\mu)$ norm as well). Use this, along with the facts that $|p_n(z)| \leq 2$ for all $z \in \D$ and the density of $\C[z]$ in $\mathcal{H}^2(\mu)$, to see that
 $$ (1-z) q_n h= p_n h \to h$$ in $\mathcal{H}^{2}(\mu)$ norm for all $h \in \mathcal{H}^2(\mu)$. Now, if $h \in \mathcal{M}$ then
$$(1-z)q_n(z)h \in (1-z)\mathcal{M}$$ and the result follows.
\end{proof}

\begin{Corollary}
Let $\mathcal{M} \subset H^2$ be a closed invariant subspace for the Ces\`{a}ro operator $C$. Then $C \mathcal{M}$ is dense in $\mathcal{M}$.
\end{Corollary}

Recall that a  vector $\vec{x}$ in a Hilbert space $\mathcal{H}$ is a {\em cyclic vector} for a bounded operator $T$ on $\mathcal{H}$ if
$$\overline{\operatorname{span}}\{T^{n} \vec{x}: n \geq 0\} = \mathcal{H}.$$

From Theorem \ref{zbzbssjjquUU} we know that the operator $I-C$, when  restricted to $(u_\alpha H^2)^\perp$,
is similar to the operator $M_b$ of multiplication by $b(s)=(s-1/2)/(s+1/2)$ on $H^2(\C^+)$.
From this we can determine some cyclic vectors for $C$ when restricted to the model spaces $(u_{\alpha} H^2)^{\perp}$, $\alpha > 0$.

\begin{Proposition}\label{prop:cyclo}
For every $\lambda \in \C^+$ the function
$$\frac{1}{s+\lambda}$$
is a cyclic vector for $M_b$ on $H^2(\C^+)$. Moreover, for every inner function $\Theta \in H^\infty(\C^+)$
the function $$\frac{\Theta(s)}{s+\lambda}$$ is cyclic for the restriction of $M_b$
to $\Theta H^2(\C^+)$.
\end{Proposition}
\begin{proof}
There is an orthonormal basis of $H^2(\C^+)$ given by
\[
\frac{1}{\sqrt{2\pi}}
\frac{(s-\frac{1}{2})^n}{(s+\frac{1}{2})^{n+1}},
 \quad n \geq 0.
\]
This can be seen most easily by transforming to $\D$ using the conformal mapping
$$s \mapsto \frac{s - \frac{1}{2}}{s + \frac{1}{2}}.$$ With respect to this basis, $M_b$ is is a unilateral shift
(so we see again that its spectrum is $\overline\D$). This tells us that
 $1/(s+1/2)$ is cyclic, as a nonzero multiple of the first
basis vector.
Now a vector remains cyclic for $M_b$  when multiplied by an invertible function in $H^\infty(\C^+)$,
and in particular,
$$\frac{s+\frac{1}{2}}{s+\lambda}.$$ (We have many other choices but these ones
lead to simpler expressions.)
The result for $\Theta H^2(\C^+)$ follows easily.
\end{proof}

We now track the vectors from Proposition \ref{prop:cyclo} back to the original $H^2$ setting. For $\alpha>0$  write $a=\log \alpha$.
There is a surjective isomorphism $L_a$ between the spaces
$L^2((-\infty,a), w(y) \, dy)$ and $H^2(\C^+)$ given by
\[
(L_a f)(s)= \int_{-\infty}^a f(y) e^{-s(a-y)} \, dy.
\]
Taking $f(y)=e^{\lambda y}$, we conclude from
Proposition \ref{prop:cyclo} that this is a cyclic vector for
the operator $L_a^{-1}M_bL_a$ on $L^2(-\infty,a)$ corresponding to $I-C$ (and hence for the operator corresponding to $C$) as in Theorem \ref{zbzbssjjquUU}.
Note that, since we are restricting $w$ to $(-\infty,a)$,
on which it is bounded above and below (Figure \ref{Figure_weight}), we can drop all
references to $w$ from now on.

We now recall the operator $(Th)(x)=x^{-1/2}h(\log x)$ from \eqref{TTTTT} to transfer this to
$L^2(0,\alpha)$ and obtain the cyclic vectors
$$f_\lambda(x)=x^{-\frac{1}{2}} x^{\lambda}.$$

Let us calculate the Laplace transform of $f_\lambda \in L^2(0, \alpha)$, i.e.,
\[
\int_0^\alpha x^{-1/2} x^{\lambda} e^{-sx} \, dx
\]
to obtain a cyclic vector for the unitarily equivalent form of $C$ on the model space
$(e^{-\alpha s}H^2(\C^+))^\perp$.
The answer is too complicated to be usable for most values of $\lambda$ but if we take $\lambda=\tfrac{1}{2}$
we obtain
$$\dfrac{1-e^{-\alpha s}}{s}.$$

Finally,
using $\mathcal{U}^{-1}$ defined in \eqref{UUUinverse}, we obtain a cyclic vector for $C$ on $(u_\alpha H^2)^\perp$, namely
\[
g_\alpha(z)= \frac{1-u_\alpha(z)}{1+z} .
\]
The singularity at $-1$ for $g_{\alpha}$ is removable. Also note that $g_{\alpha}$ belongs to $(u_\alpha H^2)^\perp$ since
if $h=(z+1)k \in (z+1)H^2$  then
\[
\langle g_\alpha , u_\alpha h \rangle_{H^2} =\langle 1-u_\alpha,  u_\alpha z k \rangle_{H^2} =0,
\]
and this is true for a dense set of $h$ hence for all $h \in H^2$. We summarize this discussion with the following.

\begin{Proposition}\label{cyclicCCC}
For $\alpha > 0$, the vector
$$\frac{1 - u_{\alpha}(z)}{1 + z}$$ belongs to $(u_{\alpha} H^2)^{\perp}$ and is a cyclic vector for the restriction of $C$  to $(u_{\alpha} H^2)^{\perp}$.
\end{Proposition}

Even though $C|_{\mathcal{M}}$, where $\mathcal{M}$ is a closed $C$-invariant subspace contained in $(u_{\alpha} H^2)^{\perp}$, is cyclic, it seems that finding cyclic vectors cannot be done explicitly, except perhaps for some very simple inner functions $\Theta$ in Proposition \ref{prop:cyclo}.

For $w \in \D$ and
$$\lambda=\frac{\bar w}{1-\bar w} \in \{z \in \C: \re z > -\tfrac{1}{2}\},$$
we want an expression for
$$\langle g_\alpha, (1-z)^\lambda \rangle_{H^2},$$
the cyclic vector for $M_{z}$ when restricted to $K((u_{\alpha} H^2)^{\perp})$.
Since $\mathcal{U}$ from \eqref{UUU} is a unitary  operator, and thus preserves inner products, we may do the calculation
in $H^2(\C^+)$. Thus, to within irrelevant constants, we
obtain
\[
\left\langle \frac{1-e^{-\alpha s}}{s}, \frac{2^\lambda}{(s+1)^{\lambda+1}} \right\rangle_{H^2(\C^+)}.
\]
The first term in the inner product is the Laplace transform of the characteristic function
of $(0,\alpha)$, as we have seen already, while the second is the Laplace transform
of the function
$$\frac{2^\lambda e^{-t} t^\lambda}{\Gamma(\lambda+1)},$$
where $\Gamma$ is the standard gamma function.

Thus, since the normalized Laplace transform is unitary, we obtain the
 Kriete--Trutt cyclic vectors given by the function
\[
U_\alpha (w) := \frac{2^{w/(1-w)}}{\Gamma(1/(1-w))}\int_0^\alpha e^{-t} t^{w/(1-w)} \, dt, \quad w \in \D.
\]
This is summarized with the following.

\begin{Proposition}\label{117yyyY77yRR}
For $\alpha > 0$, the function $U_{\alpha}$ is a cyclic vector for $M_z$ when restricted to $K((u_{\alpha} H^2)^{\perp})$.
\end{Proposition}

The function $U_{\alpha}$ does not have a usable formula. However, we can say a few things.

\begin{Proposition}
The function $U_{\alpha}$ has no zeros on $\D$.
\end{Proposition}

\begin{proof}
This will follow from the general fact that if $f \in (u_{\alpha} H^2)^{\perp}$ and
$$\overline{\operatorname{span}}\{C^n f: n \geq 0\} = (u_{\alpha} H^2)^{\perp},$$
then $K f$,
which will be a generator for the $M_z$-invariant subspace  $K( (u_{\alpha} H^2)^{\perp})$, has no zeros in $\D$.

Indeed, suppose $(Kf)(\lambda) = 0$ for some $\lambda \in \D$. Since $K C f(z) = (1 - z) K f(z)$ (Proposition \ref{KTTTT}), every function in
$$\overline{\operatorname{span}}\{(1 - z)^n K f(z): n \geq 0\}$$ would have a zero at $\lambda$ (recall \eqref{RKHS}). The previous identity would say that every function in
$$\overline{\operatorname{span}}\{K C^n f: n \geq 0\}$$ vanishes  at $\lambda$.
However, $f$ is a cyclic vector for $C|_{(u_{\alpha} H^2)^{\perp}}$ and so every function from $K((u_{\alpha} H^2)^{\perp})$ would vanish at $\lambda$. It follows that
$$\langle f, (1 - w)^{\frac{\overline \lambda}{1 - \overline\lambda}}\rangle_{H^2} = 0 \quad \mbox{for all $f \in (u_{\alpha} H^2)^{\perp}$}$$ which, in turn, implies that
$$(1 - z)^{\frac{\overline \lambda}{1 - \overline\lambda}} \in u_{\alpha} H^2.$$
This last statement is impossible since $(1 - z)^{\frac{\overline \lambda}{1 - \overline\lambda}}$ is an outer function and this can  not belong to any Beurling subspace $u_{\alpha} H^2$.
\end{proof}

\begin{Proposition}
The function $U_{\alpha}$ does not belong to $H^2$.
\end{Proposition}

\begin{proof}
From Proposition \ref{nosthnehrenaasdd}, we know that $K((u_{\alpha} H^2)^{\perp})$ is never equal to $\overline{u \mathcal{H}^2(\mu)}$ for any inner function $u$. If $U_{\alpha} \in H^2$, then, since $\{p U_{\alpha}: p \in \C[z]\}$ is dense in $K((u_{\alpha} H^2)^{\perp})$ (recall Proposition \ref{117yyyY77yRR}), then \cite[Thm.~5]{MR350489} (mentioned above) would say that $K((u_{\alpha} H^2)^{\perp}) = \overline{u \mathcal{H}^2(\mu)}$, which we know is not the case.
\end{proof}

\section{A connection to universal operators}\label{Careadsdfds}

A bounded operator $U$ on a Hilbert space $\mathcal{H}$ is {\em universal} for $\mathcal{H}$ if given any bounded operator $T$ on $\mathcal{H}$, there exists a constant $a \not = 0$ and a closed invariant subspace $\mathcal{M}$ for $U$ such that $U|_{\mathcal{M}}$ is similar to $a T$. Universal operators thus have an extraordinarily rich class of invariant subspaces. Though universal operators might not seem to exist at all, a theorem of Caradus \cite{MR250104} (see also \cite[Ch.~8]{MR2841051}) shows they exist in abundance. Indeed, if $\mathcal{H}$ is an infinite dimensional  separable Hilbert space and $U$ is a bounded operator on $\mathcal{H}$ such that $\ker(U)$ is infinite dimensional and $U$ is surjective, then $U$ is universal.

For example, if $u$ is an inner function that is not a finite Blaschke product, then the co-analytic Toeplitz operator $T_{\overline{u}}$ on $H^2$ satisfies Caradus' criterion and is thus universal \cite[Prop.~16.7.1]{MR4545809}.
Although neither $C$ nor $C^*$ is universal (because they are injective),
we do have the following.

\begin{Theorem}\label{uuuuniverereersal}
There exists a bounded analytic function $F$ on the disk $D(1, 1)$ such that $F(C^{*})$ is universal.
\end{Theorem}

Note that $C^{*}$ has an $H^{\infty}$ functional calculus (see the remarks in \S \ref{ramrekresas}) and so the operator $F(C^{*})$ makes sense. Our proof of Theorem \ref{uuuuniverereersal} uses a universality result for composition operators from \cite{MR4373152}, an analysis from \cite{MR310691}, and a commutant result from \cite{MR287352}.

The following theorem is from Carmo and Noor \cite[Theorem 4.3]{MR4373152}.

\begin{Theorem}\label{Noor}
Let $\phi$ be a hyperbolic non-automorphism of $\D$ with a fixed point $\zeta\in \T$ and the other outside the closed unit disk $\overline{\D}$, possibly at $\infty$. If $a:=\phi'(\zeta)\in (0, 1)$, then for each $\lambda$ with $0<|\lambda |<a^{-1/2}$ the operator $C_{\phi}-\lambda I$ is universal on the Hardy space $H^2$.
\end{Theorem}

In our case, each of the symbols in the semigroup
$
\phi_t(z)= e^{-t}z + 1 - e^{-t}, t>0,
$
from \S \ref{two}
induces a composition operator with  a universal translate. Here $a_t=\phi_t'(1)= e^{-t}$ and the fixed points are $1 \in \T$ and at $\infty$.
Clearly, the lattice of invariant subspaces of $C_{\phi_t}-\lambda I$ and $C_{\phi_t}$ are the same.

\begin{proof}[Proof of Theorem \ref{uuuuniverereersal}]
To help with the typesetting, fix $t > 0$ and let $\alpha = e^{-t}$. The composition operator induced by $\phi_t$  becomes
$C_{\alpha z + (1 - \alpha)}$ and, with respect to the usual orthonormal basis $\{z^n\}_{n \geq 0}$ for $H^2$, has the matrix representation
$$\begin{bmatrix}
1 & \alpha & \alpha^2 & \alpha^3 & \cdots\\[3pt]
\0 & (1 - \alpha) & 2\alpha (1 - \alpha)& 3 \alpha^2 (1 - \alpha) & \cdots\\[3pt]
\0 & \0 & (1 - \alpha)^2 & 3 \alpha (1 - \alpha)^2 & \cdots \\[3pt]
\0 & \0 & \0 & (1 - \alpha)^3 & \cdots \\
\vdots &\vdots\ & \vdots & \vdots &\ddots
\end{bmatrix}.$$
See a paper of Deddens \cite[p.~862]{MR310691} for the details of this.
Note that $C^{*}$  has the matrix representation
$$ \begin{bmatrix}
1 & \frac{1}{2} & \frac{1}{3} & \frac{1}{4} & \frac{1}{5} &  \cdots\\[3pt]
\0 & \frac{1}{2} & \frac{1}{3} & \frac{1}{4} & \frac{1}{5} & \cdots\\[3pt]
\0 & \0 & \frac{1}{3} & \frac{1}{4} & \frac{1}{5} & \cdots\\[3pt]
\0 & \0 & \0 & \frac{1}{4} & \frac{1}{5} & \cdots\\[3pt]
\0 & \0 & \0 & \0 & \frac{1}{5} & \cdots\\[-3pt]
\vdots & \vdots & \vdots & \vdots & \vdots & \ddots
\end{bmatrix}.$$
One can check, as Deddens did, that the two operators above commute. By Theorem \ref{Noor} there exists a constant $\beta_{\alpha}$ for which
$$C_{\alpha z + (1 - \alpha)} - \beta_{\alpha} I$$
is universal. Moreover, this operator also commutes with $C^{*}$ and has the matrix representation
$$\begin{bmatrix}
1 - \beta_{\alpha} & \alpha & \alpha^2 & \alpha^3 & \cdots\\[3pt]
\0 & (1 - \alpha) - \beta_{\alpha} & 2\alpha (1 - \alpha)& 3 \alpha^2 (1 - \alpha) & \cdots\\[3pt]
\0 & \0 & (1 - \alpha)^2  - \beta_{\alpha} & 3 \alpha (1 - \alpha)^2 & \cdots \\[3pt]
\0 & \0 & \0 & (1 - \alpha)^3 - \beta_{\alpha} & \cdots \\
\vdots &\vdots\ & \vdots & \vdots &\ddots
\end{bmatrix}.$$
By a result of Shields and Wallen \cite{MR287352} (which describes the commutant of $C$ -- and thus $C^{*}$), there is a bounded analytic function $F$ on $D(1, 1)$ such that
$$F(C^{*}) = C_{\alpha z + (1 - \alpha)} - \beta_{\alpha} I.$$
This says that $F(C^{*})$ is universal.
\end{proof}

By a trick from \cite[p.~863]{MR310691}, we can actually compute $F$. Note that $$F(C^{*}) = \begin{bmatrix}
F(1) & * & * & * & * &  \cdots\\[3pt]
\0 & F(\frac{1}{2}) &* & * & * & \cdots\\[3pt]
\0 & \0 & F(\frac{1}{3}) & * & * & \cdots\\[3pt]
\0 & \0 & \0 & F(\frac{1}{4}) & * & \cdots\\[3pt]
\0 & \0 & \0 & \0 & F(\frac{1}{5}) & \cdots\\[-3pt]
\vdots & \vdots & \vdots & \vdots & \vdots & \ddots
\end{bmatrix},$$
where the $*$ entries in the matrix above are not important. Comparing the diagonal entries of the matrix representations of $C_{\alpha z + (1 - \alpha)} - \beta_{\alpha} I$ and $F(C^{*})$, we see that
$$F(\tfrac{1}{n}) = (1 - \alpha)^{n - 1} - \beta_{\alpha} \quad \mbox{for all $n \geq 1,$}$$ and thus, since $\{\tfrac{1}{n}\}_{n \geq 1}$ is not a Blaschke sequence for the disk $D(1, 1)$, we see that
$$F(z) = (1 - \alpha)^{1/z - 1} - \beta_{\alpha}, \quad z \in D(1, 1).$$

\begin{Remark}
The lattice of invariant subspaces for $C^*$ is strictly contained in the corresponding lattice
for $F(C^*)$. For the function $F$ is not injective on $\sigma_{p}(C^{*}) = D(1, 1)$ (Proposition \ref{BrownShields}),
and so there exist distinct eigenvalues $\lambda$, $\mu$ for $C^*$, with corresponding
eigenvectors $g, h \in H^2$, such that $F(\lambda)=F(\mu)$. Now the one-dimensional space $\mathcal E$
spanned by
$g+h$ is not invariant under $C^*$, but $F(C^*)(g+h)=F(\lambda)g+F(\mu)h$ and
so $\mathcal E$ is an invariant subspace for $F(C^*)$.
\end{Remark}

\section{Some final remarks}\label{ramrekresas}

So how complicated is the invariant subspace structure for the Ces\`{a}ro operator? Since the invariant subspaces for the Ces\`{a}ro operator are in one-to-one (and order preserving) correspondence with the invariant subspaces for the multiplication operator $M_z$ on the Kriete--Trutt space $\mathcal{H}^2(\mu)$, the complexity of the invariant subspaces of the the Ces\`{a}ro operator is reflected in the complexity of the $M_z$-invariant subspaces of $\mathcal{H}^2(\mu)$.

So what kind of space is $\mathcal{H}^2(\mu)$? 
In some ways it is more like the Hardy space $H^2$ than the Beurling space 
$A^2$ of analytic functions on $\D$ which are square integrable with respect to area measure \cite{MR2033762}, since, by the main theorem of \cite{MR2480609},
$\operatorname{dim}(\mathcal{M}/z \mathcal{M}) = 1$ for any nonzero $M_z$-invariant subspace $\mathcal{M}$ of $H^2(\mu)$. 
On the other hand, there are  $M_z$-invariant subspaces which are Beurling-like, and of the form $\overline{u \mathcal{H}^2(\mu)}$, and those which are not.

Some further parallels between the Hardy shift, the Bergman shift, and the operator $T=I-C$
can be observed as follows:
\begin{itemize}
\item $\|T\|=\|T^*\|=1$ and $\sigma(T)=\sigma(T^*)=\overline\D$ (Proposition \ref{BrownShields});
\item Using the matrix representations of $C$ and $C^{*}$ from \eqref{Cedcecefcematrix}, a calculation shows that  $$TT^*=
\begin{bmatrix}
0 & \0 & \0 &  \0 & \cdots\\[3pt]
\0 & \frac{1}{2} & \0 & \0 & \cdots\\[3pt]
\0 & \0 & \frac{2}{3} & \0 & \cdots\\[3pt]
\0 & \0 & \0 & \frac{3}{4} & \cdots\\[3pt]
\vdots & \vdots & \vdots & \vdots & \ddots\\
\end{bmatrix}
$$and so $\|T^* \vec{x}\|<\|\vec{x}\|$ for all $\vec{x} \in \ell^2$.
Hence $T^*$ is completely non-unitary.
\item By    \cite[Prop.\ 4.6]{MR0787041}, the   properties above imply that $T$ (and hence $T^*$) has an isometric $H^\infty$
functional calculus. That is, the operators lie in the class $\A$. 
\item As shown in \cite{MR0955549}, every operator of class $\A$ is reflexive, which
means that every operator fixing all $T$'s invariant subspaces can be approximated in
the weak operator topology by polynomials in $T$. This is a property that
 guarantees a rich subspace lattice.
\end{itemize}

There are other curious facts from  \cite{MR350489} which only add to the mystery of whether or not $\mathcal{H}^2(\mu)$ is closer to a Hardy space or a Bergman space. If $f \in \mathcal{H}^2(\mu) \setminus \{0\}$ vanishes on a sequence $\{z_n\}_{n \geq 1}$  in $\D$ which does not accumulate at the point $1$, then $\{z_n\}_{n \geq 1}$ must satisfy the Blaschke condition $\sum_{n \geq 1} (1 - |z_n|) < \infty$ (which seems to point towards $\mathcal{H}^2(\mu)$ being closer to $H^2$). Moreover, if $J$ is a closed arc of $\T$ which does not contain $1$ and $f \in \mathcal{H}^2(\mu) \setminus \{0\}$, then $\int_{J} \log |f| dm > -\infty$ \cite[Thm.~2.2]{Duren} (again pointing towards $\mathcal{H}^2(\mu)$ being closer to $H^2$). However, there are $f \in \mathcal{H}^2(\mu) \setminus \{0\}$ which vanish on sequences $\{z_n\}_{n \geq 1}$ which do not satisfy the Blaschke condition (pointing towards $\mathcal{H}^2(\mu)$ being closer to $A^2$). 

Furthermore \cite[Lemma 2]{MR281025}, the functions
$$g_{0}(z) = 1;$$
$$g_{n}(z)  = \frac{1}{(1 - z)^n} z (z - \tfrac{1}{2}) \cdots (z - \tfrac{n - 1}{n}), \quad n \geq 1,$$
form an orthonormal basis for $\mathcal{H}^2(\mu)$. In particular,
$$\frac{1}{(1 - z)^n} \in \mathcal{H}^2(\mu) \quad \mbox{for all $n \geq 1$}$$
which shows that $\mathcal{H}^2(\mu)$ is a space much bigger than $H^2$ (and even bigger than $A^2$). It is often the case that ``large'' Hilbert spaces of analytic functions on $\D$ have a rich class of $M_z$-invariant subspaces.

\bibliographystyle{plain}

\bibliography{references}

\end{document}